\documentclass[a4paper]{amsart}

\usepackage{amsmath}
\usepackage{amssymb}
\usepackage{amsthm}
\usepackage{latexsym}
\usepackage{enumerate}
\usepackage{prettyref}
\usepackage{hyperref}
\usepackage{tikz}
\usetikzlibrary{calc}
\usepackage[arrow,matrix]{xy}

\newtheorem{mainthm}{Theorem}
\newtheorem{theorem}{Theorem}[section]
\newrefformat{thm}{\hyperref[{#1}]{Theorem~\ref*{#1}}}
\newtheorem{definition}[theorem]{Definition}
\newrefformat{def}{\hyperref[{#1}]{Definition~\ref*{#1}}}
\newtheorem{lemma}[theorem]{Lemma}
\newrefformat{lem}{\hyperref[{#1}]{Lemma~\ref*{#1}}}
\newtheorem{proposition}[theorem]{Proposition}
\newrefformat{prop}{\hyperref[{#1}]{Proposition~\ref*{#1}}}
\newtheorem{corollary}[theorem]{Corollary}
\newrefformat{cor}{\hyperref[{#1}]{Corollary~\ref*{#1}}}
\newtheorem{remark}[theorem]{Remark}
\newrefformat{rem}{\hyperref[{#1}]{Remark~\ref*{#1}}}
{
\newtheorem{examplecore}[theorem]{Example}}
\newrefformat{ex}{\hyperref[{#1}]{Example~\ref*{#1}}}

\newcommand{\colim}{\ensuremath{\operatornamewithlimits{colim}}}

\newcommand{\Stab}{\ensuremath{\operatorname{Stab}}}
\newcommand{\Aut}{\ensuremath{\operatorname{Aut}}}
\newcommand{\End}{\ensuremath{\operatorname{End}}}

\newcommand{\Pic}{\ensuremath{\operatorname{Pic}}}

\newcommand{\op}{\operatorname}

\begin{document}

\title[Parabolic subcomplexes and homology of $\op{SL}_2$]{Homology of
  $\op{SL}_2$ over function fields I: \\ parabolic subcomplexes}  
\author{Matthias Wendt}
\date{April 2014}

\thanks{This work has been partially supported by the
  Alexander-von-Humboldt-Stiftung.}
\address{Matthias Wendt, Fakult\"at Mathematik,
Universit\"at Duisburg-Essen, Thea-Leymann-Strasse 9, 45127, Essen,
Germany} 
\email{matthias.wendt@uni-due.de}


\begin{abstract}
The present paper studies the homology of the groups $\op{SL}_2(k[C])$
and $\op{GL}_2(k[C])$, where
$C=\overline{C}\setminus\{P_1,\dots,P_s\}$ is a smooth affine curve
over an algebraically closed field $k$. 
It is well-known that these groups act on a product of trees and the
quotients can be described in terms of certain equivalence classes of
rank two vector bundles on the curve $\overline{C}$. There is a
natural subcomplex consisting of cells with non-trivial isotropy
group. The paper provides explicit formulas for the equivariant
homology of this ``parabolic subcomplex''. 
These formulas also describe homology of $\op{SL}_2(k[C])$ above
degree $s$, generalizing a result of Suslin for the case $s=1$.
\end{abstract}

\maketitle
\setcounter{tocdepth}{1}
\tableofcontents

\section{Introduction}
\label{sec:intro}

This paper is the first in a series of papers studying the homology of
rank one linear groups over function rings and function fields of
curves, mostly over algebraically closed fields. Of course, the groups
$\op{SL}_2$ and $\op{GL}_2$,  their structure, their homology and
representation theory, have been subject to a lot of research in
number-theoretic situations. However, surprisingly little information
is available  on the structure of $\op{SL}_2$ in situations where the
usual, more analytic, methods fail - for curves over infinite base
fields.   

That being said, the present paper still follows the standard path to
computation of homology of linear groups over function rings of
curves; the group $\op{SL}_2(k[C])$ acts on a product of trees, the
building $\mathfrak{X}_C$, and the isotropy spectral sequence
associated to this action can be used to obtain information on group
homology of $\op{SL}_2(k[C])$. In general, understanding the structure
of the quotient $\op{SL}_2(k[C])\backslash\mathfrak{X}_C$ is rather
difficult, and this is one of the main reasons for the lack of
group homology computations. There is, nevertheless, one part of group
homology that is easier to understand: taking inspiration from the 
number-theoretic situation, we consider a subcomplex of the building,
called \emph{parabolic subcomplex} $\mathfrak{P}_C$,
cf. \prettyref{def:pc}, consisting of cells with non-unipotent
stabilizer. The equivariant homology of this subcomplex is 
a function field analogue of Farrell-Tate homology, and sits in a long
exact sequence, cf. \prettyref{prop:split}
$$
\cdots\to\op{H}_\bullet^{\op{SL}_2(k[C])}(\mathfrak{P}_C)\to
\op{H}_\bullet(\op{SL}_2(k[C])) \to
\op{H}^{\op{SL}_2(k[C])}_\bullet(\mathfrak{U}_C)\to\cdots,
$$
where $\op{H}^{\op{SL}_2(k[C])}_\bullet(\mathfrak{U}_C)$ is an analogue of
cuspidal homology (or homology of the Steinberg module) in the
number-theoretic situations.  
Moreover, the equivariant homology of the parabolic subcomplex  can be
computed very explicitly in terms of the homology of (normalizers of)
maximal tori and refined scissors congruence groups
$\mathcal{RP}^1_\bullet(k)$, cf. \prettyref{sec:rp1}. The following is
the main result of the  paper; it describes the equivariant homology
of the parabolic subcomplex for $\op{SL}_2(k[C])$ with $k$  an algebraically
closed field. For the proofs, cf. \prettyref{lem:conn},
\prettyref{prop:hlgytorus} and \prettyref{prop:hlgynorm}.

\begin{mainthm}
\label{thm:mainsl2}
Let $k$ be an algebraically closed field, let $\overline{C}$ be a
smooth projective curve over $k$, let $P_1,\dots,P_s\in \overline{C}$
be closed points, and set $C=\overline{C}\setminus\{P_1,\dots,P_s\}$. 
Denoting by $\mathfrak{P}_C$ the \emph{parabolic subcomplex} of the
building $\mathfrak{X}_C$, cf. \prettyref{def:pc}, we have the
following formulas for the $\op{SL}_2(k[C])$-equivariant homology of
$\mathfrak{P}_C$: 
\begin{enumerate} 
\item The connected components of the parabolic subcomplex
  $\mathfrak{P}_C$ are indexed by the quotient set
  $\mathcal{K}(C)=\Pic(C)/\iota$ of the Picard   group of $C$ modulo
  the involution $\iota:\mathcal{L}\mapsto \mathcal{L}^{-1}$.  
The result is a direct sum decomposition: 
  $$
\op{H}_\bullet^{\op{SL}_2(k[C])}(\mathfrak{P}_C,\mathbb{Z}[1/2])\cong
  \bigoplus_{[\mathcal{L}]\in\mathcal{K}(C)}
\op{H}_\bullet^{\op{SL}_2(k[C])}(\mathfrak{P}_C(\mathcal{L}),\mathbb{Z}[1/2]). 
$$ 
\item If $[\mathcal{L}]\in\mathcal{K}(C)$ is such that
  $\mathcal{L}|_C\not\cong\mathcal{L}|_C^{-1}$, then the homology of
  the component $\mathfrak{P}_C(\mathcal{L})$ is the homology of the
  group $k[C]^\times$:
$$
\op{H}_\bullet^{\op{SL}_2(k[C])}(\mathfrak{P}_C(\mathcal{L}),\mathbb{Z}[1/2])\cong 
\op{H}_\bullet(k[C]^\times,\mathbb{Z}[1/2]).
$$
\item If $[\mathcal{L}]\in\mathcal{K}(C)$ is such that
  $\mathcal{L}|_C\cong\mathcal{L}|_C^{-1}$, then there is a long exact
  sequence 
$$
\cdots\to
\mathcal{RP}^1_{i+1}(k)\otimes_{\mathbb{Z}}
\mathbb{Z}[1/2,k[C]^\times/(k[C]^\times)^2]\to 
\op{H}_i(\tilde{\mathcal{SN}},\mathbb{Z}[1/2])\to 
$$
$$
\to
\op{H}_i^{\op{SL}_2(k[C])}(\mathfrak{P}_C(\mathcal{L}),\mathbb{Z}[1/2])
\to 
\mathcal{RP}^1_i(k)\otimes_{\mathbb{Z}}
\mathbb{Z}[1/2,k[C]^\times/(k[C]^\times)^2]\to
\cdots, 
$$
where $\tilde{\mathcal{SN}}$ denotes the group of monomial matrices in
$\op{SL}_2(k[C])$ and $\mathcal{RP}^1_i(k)$ are refined scissors
congruence groups, cf. \prettyref{sec:rp1}. 
\end{enumerate}
\end{mainthm}

There is a similar result for the equivariant homology of the
parabolic subcomplex for the group $\op{PGL}_2$,
cf. \prettyref{prop:hlgytorus} and \prettyref{prop:hlgynormgl2}. 

The result is proved by explicitly computing the quotient
$\op{SL}_2(k[C])\backslash\mathfrak{P}_C$ and then working out the
isotropy spectral sequence in detail. The analysis of the spectral
sequence is significantly 
simplified by working with $\mathbb{Z}[1/2]$-coefficients, and it is
not clear what modifications would be necessary to get a $2$-integral
result. The exact sequence appearing in \prettyref{thm:mainsl2},
point(3), above is induced 
from the one connecting group homology of $\op{SL}_2(k)$ to group
homology of the normalizer of the maximal torus in $\op{SL}_2(k)$ and
refined scissors congruence groups $\mathcal{RP}^1_i$. The latter
exact sequence would appear to be well-known, cf. \cite[Chapters 8,
15]{dupont}, but for the sake of completeness we discuss the
definition of the groups $\mathcal{RP}^1_i$ as well as a proof of the
abovementioned exact sequence in \prettyref{sec:rp1}.  

It has been known for some time, and surfaced particularly in the
recent work of Kevin Hutchinson \cite{hutchinson:rb,hutchinson:bw},
that the action of square classes $F^\times/(F^\times)^2$ is an
extremely helpful tool in understanding group homology of
$\op{SL}_2(F)$. The result above exhibits a complete description as
well as a geometric interpretation of the square-class action on the
parabolic part of group homology.  

The above computations of the equivariant homology
of the parabolic subcomplex also imply formulas for group homology of
$\op{SL}_2(k[C])$ above the dimension of the product of trees. The
next result follows directly from a mod $\ell$-version of
\prettyref{thm:mainsl2}  and \prettyref{prop:split}. 

\begin{corollary}
Let $k$ be an algebraically closed field, let $\overline{C}$ be a
smooth projective curve over $k$, and set
$C=\overline{C}\setminus\{P_1,\dots,P_s\}$.  Let $\ell$ be an odd prime
different from the characteristic of $k$. 
For $i>s$, there are isomorphisms
$$
\op{H}_i^{\op{SL}_2(k[C])}(\mathfrak{P}_C,\mathbb{Z}/\ell)\cong
\op{H}_i(\op{SL}_2(k[C]),\mathbb{Z}/\ell).
$$
In particular, we have a direct sum decomposition for mod $\ell$
group homology of $\op{SL}_2(k[C])$ above homological degree $s$
\begin{eqnarray*}
\op{H}_{\bullet>s}(\op{SL}_2(k[C]),\mathbb{Z}/\ell)&\cong&
  \bigoplus_{[\mathcal{L}]\in\mathcal{K}(C),\mathcal{L}|_C\not\cong\mathcal{L}|_C^{-1}}
\op{H}_{\bullet>s}(k[C]^\times,\mathbb{Z}/\ell)\\&\oplus &
\bigoplus_{[\mathcal{L}]\in\mathcal{K}(C),\mathcal{L}|_C\cong\mathcal{L}|_C^{-1}}
\op{H}_{\bullet>s}^{\op{SL}_2(k[C])}(\mathfrak{P}_C(\mathcal{L}),\mathbb{Z}/\ell)
\end{eqnarray*}
as well as an exact sequence describing the normalizer terms: 
$$
\cdots\to\mathcal{RP}^1_{i+1}(k,\mathbb{Z}/\ell)[k(C)^\times/(k(C)^\times)^2]\to 
\op{H}_{i}(\tilde{\mathcal{SN}},\mathbb{Z}/\ell) \to$$
$$
\op{H}_{i}^{\op{SL}_2(k[C])}(\mathfrak{P}_C(\mathcal{L}),\mathbb{Z}/\ell)\to 
\mathcal{RP}^1_{i}(k,\mathbb{Z}/\ell)[k(C)^\times/(k(C)^\times)^2]\to\cdots
$$
\end{corollary}

Note that the above exact sequence is not simply the exact sequence of
\prettyref{thm:mainsl2} tensored with finite coefficients, but using a
mod $\ell$-version of the groups $\mathcal{RP}^1_\bullet(k)$,
cf. \prettyref{sec:rp1}. 

The above theorem provides a generalization of a theorem of Suslin,
cf. \cite[Theorem 4.5.7]{knudson:book}; for $s=1$ the above
formula reduces to the one in loc.cit.: 
\begin{eqnarray*}
\op{H}_{\bullet>1}(\op{PGL}_2(k[C]),\mathbb{Z}/\ell)
&\cong&\bigoplus_{\mathcal{L}\in \mathcal{K}(C), 2\mathcal{L}=0}
\op{H}_{\bullet>1}(\op{PGL}_2(k),\mathbb{Z}/\ell)\\&\oplus &
\bigoplus_{\mathcal{L}\in\mathcal{K}(C), 2\mathcal{L}\neq 0}
\op{H}_{\bullet>1}(k^\times,\mathbb{Z}/\ell).
\end{eqnarray*}
Note that the finite coefficients are necessary: on the
curve $\overline{C}$, there may exist bundles with unipotent
automorphism group (like Atiyah's bundles $\mathcal{F}_2$ on an
elliptic curve). The automorphism groups influence the homology in
degrees above $s$, but are not accounted for in the parabolic
subcomplex $\mathfrak{P}_C$. They do, however, not influence the mod
$\ell$ homology above degree $s$ because the additive group $(k,+)$ is
uniquely $\ell$-divisible.

The definition of the parabolic subcomplex is functorial in the curve,
hence morphisms of curves induce morphisms on equivariant homology of
the corresponding parabolic subcomplexes. The induced morphisms can be
explicitly described and, for a function field $k(C)$, allow to define
a ``parabolic homology'' of $SL_2(k(C))$ via the obvious limit
process. In analogy to the notation for Farrell-Tate homology, we
denote 
$$
\widehat{\op{H}}_\bullet(\op{SL}_2(k(C)),\mathbb{Z}/\ell)=
\colim_{S\subseteq\overline{C}(k)}
\op{H}_\bullet^{\op{SL}_2(k[\overline{C}\setminus
  S])}(\mathfrak{P}_{\overline{C}\setminus S},\mathbb{Z}/\ell),
$$
where the index set of the colimit above is the set of finite sets
of closed points of $\overline{C}$, ordered by inclusion. 
The second main result of the paper now provides a formula for the
parabolic homology of $\op{SL}_2(k(C))$ and establishes a rigidity
result for it. It may not come as a surprise, but taking the limit to
the algebraic closure  of a function field,  the parabolic homology
has exactly the form predicted by the Friedlander-Milnor
conjecture, cf. \cite{friedlander:mislin}. The result follows from
\prettyref{prop:hlgyfun} and the 
above \prettyref{thm:mainsl2} resp. its corollary, cf. \prettyref{prop:fthlgy}.

\begin{mainthm}
\label{thm:main2}
Let $k$ be an algebraically closed field, let $C$ be a smooth curve
over $k$ and let $\ell$ be an odd prime different from the characteristic
of $k$. Then we have the following exact sequence 
$$
\cdots\to
\mathcal{RP}^1_{i+1}(k,\mathbb{Z}/\ell)[k(C)^\times/(k(C)^\times)^2]\to 
\op{H}_i(\op{N}(k(C)),\mathbb{Z}/\ell)\to 
$$
$$
\to
\widehat{\op{H}}_i(\op{SL}_2(k(C)),\mathbb{Z}/\ell)
\to 
\mathcal{RP}^1_i(k,\mathbb{Z}/\ell)[k(C)^\times/(k(C)^\times)^2]\to
\cdots, 
$$
where $\op{N}(k(C))$ denotes the normalizer of a maximal torus in
$\op{SL}_2(k(C))$. 

\begin{enumerate}
\item 
All classes of $\widehat{\op{H}}_i(\op{SL}_2(k(C)),\mathbb{Z}/\ell)$
become constant over the quadratic closure of $k(C)$. 
\item 
Assume $\mathcal{RP}^1_\bullet(k)=0$, i.e., Friedlander's generalized
isomorphism conjecture is true for $\op{SL}_2$ over $k$. Then 
\begin{eqnarray*}
\widehat{\op{H}}_\bullet(\op{SL}_2(\overline{k(C)}),\mathbb{Z}/\ell)&:=&
\colim_{K/k(C)
  \op{ finite}}\widehat{\op{H}}_\bullet(\op{SL}_2(K),\mathbb{Z}/\ell)
\\
&\cong &
\op{H}_\bullet(\op{N}(\overline{k(C)}),\mathbb{Z}/\ell)
\end{eqnarray*}
has exactly the form predicted by Friedlander's isomorphism conjecture
for $\op{SL}_2$ over $\overline{k(C)}$. 
In particular, Friedlander's generalized isomorphism conjecture for
$\op{SL}_2$ over $\overline{k(C)}$ is equivalent to vanishing of 
$$
\colim_{C/k}\op{H}_\bullet^{\op{SL}_2(k[C])}(\mathfrak{U}_C,\mathbb{Z}/\ell),
$$
where the colimit runs over all smooth affine curves over $k$. 
\end{enumerate}
\end{mainthm}

The above result provides two reformulations of Friedlander's
isomorphism conjecture, one as a divisibility result for ``the limit
of cuspidal homology'', and one as a detection result of homology on
the normalizer of the maximal torus. The second reformulation in turn
is close to \cite[Corollary 5.2.10]{knudson:book}. In general, we see
that the parabolic homology
$\widehat{\op{H}}_\bullet(\op{SL}_2(k(C)),\mathbb{Z}/\ell)$ describes
exactly the part of the  homology of $\op{SL}_2(k(C))$ which can be
detected on the normalizer $\op{N}(k(C))$ and some additional
subgroups of $\op{SL}_2(k(C))$ isomorphic to
$\op{SL}_2(k)$. Moreover, we see that the parabolic homology satisfies
a much stronger rigidity than expected by Friedlander's isomorphism
conjecture: classes in the parabolic homology become trivial already
over quadratically closed fields. 
Hopefully, the above result helps shed new light on
the homology of $\op{SL}_2$ over algebraically closed fields. 

Finally, we want to mention that related (and as it turns out
structurally similar) computations in the number
field case, i.e., computations of Farrell-Tate cohomology of
$\op{SL}_2(\mathcal{O}_{K,S})$ with $\mathcal{O}_{K,S}$ a ring of
$S$-integers, are being developed in joint work with Alexander D. Rahm.  

\subsection{Structure of the paper:}
We first recall preliminaries on trees and group actions in 
\prettyref{sec:prelims}. The definition and basic properties of the
parabolic subcomplex $\mathfrak{P}_C$ are given in
\prettyref{sec:pardef}. Then \prettyref{sec:local} works out the
actions of stabilizers on links and the resulting local structure of
the quotient of the parabolic subcomplex, leading to a global
description of the structure of the quotient
$\op{SL}_2(k[C])\backslash\mathfrak{P}_C$ in
\prettyref{sec:global}. The structure of $\op{SL}_2(k[C])$-equivariant
homology of the parabolic subcomplex $\mathfrak{P}_C$ is determined in
\prettyref{sec:parhlgy}. The appendix \prettyref{sec:rp1} provides
a recollection of basic facts on the refined scissors congruence
groups $\mathcal{RP}^1_\bullet(k)$.

\subsection{Acknowledgements:}
The investigations reported in the paper started during a stay at the
De Br\'un center for computational algebra at NUI Galway in August
2012. I would like to thank Alexander D. Rahm for discussions about
Farrell-Tate cohomology and his computations with Bianchi group
\cite{rahm:transactions,rahm:jalg}, which 
shaped my understanding of the structure and possible usefulness of
the parabolic part of group homology described in the paper. I would
also like to thank Kevin Hutchinson for explanations on his
computations of homology of $\op{SL}_2$ and residue maps for refined
Bloch  groups in \cite{hutchinson:rb,hutchinson:bw} as well as
inspiring discussions on various topics related to group homology. 

\section{Preliminaries: trees, group actions, vector bundles} 
\label{sec:prelims}

In this section, we fix the notation for the paper, and recall
preliminaries on rank one linear groups and (products of) trees
associated to them. We also discuss the well-known identification of
the quotient of the building modulo the group action in terms of
vector bundles on the curve. 

We mostly follow the notation of \cite{serre:book}. 
\begin{itemize}
\item $k$ denotes a commutative field, 
\item $\overline{C}$ an irreducible smooth projective curve over $k$,
\item $C=\overline{C}\setminus\{P_1,\dots,P_s\}$ a smooth affine curve
  over $k$, with $P_1,\dots,P_s$ pairwise distinct but not necessarily
  $k$-rational points of respective degrees $d_i=[k(P_i):k]$,
\item $K=k(C)$ the function field of the curve $C$, 
\item $k[C]$ the ring of functions on the curve $C$,
\item for a closed point $Q$ of $C$, $v_Q$ denotes the corresponding 
  valuation, $\mathcal{O}_Q$ the valuation ring, $k(Q)$ its residue
  field, $\deg Q=[k(Q):k]$ its degree.
\end{itemize}

\subsection{Recollection on buildings}
\label{sec:buildings}

We recall the definition and structure of the Bruhat-Tits tree
associated to $\op{SL}_2$ over a field $K$ with a valuation $v$. For
details on Bruhat-Tits trees, cf. \cite[II.1.1]{serre:book}, for the
more general theory of buildings, cf. \cite{abramenko:brown}. All the
statements below are standard and can be found in one of these books.

\begin{definition}
Let $K$ be a field equipped with a discrete valuation $v$. We denote
by $\mathcal{O}_v$ the corresponding valuation ring with maximal ideal
$\mathfrak{m}_v$. We denote by $\pi_v$ a choice of uniformizer for
$v$. Let $V=K^2$. A \emph{lattice} $L$ in $V$ is a
finitely generated $\mathcal{O}_v$-submodule of $V$ which generates
$V$. Two lattices $L_1$ and $L_2$ are called \emph{equivalent} if
there exists $\lambda\in K^\times$ such that $\lambda L_1=L_2$. We
denote by $\Lambda=[L]$ the equivalence class of $L$. To a lattice
class $\Lambda$, we can assign a \emph{type} $v(\det\Lambda)\mod 
2\in\mathbb{Z}/2\mathbb{Z}$.  

The Bruhat-Tits tree associated to $(K,v)$ is the following simplicial
complex: the $0$-simplices are equivalence classes of
lattices. Lattice classes $\Lambda_0$ and $\Lambda_1$ are connected by
an edge if there exist representatives $L_i$ of $\Lambda_i$ such that
$\pi_v L_1\subset L_0\subset L_1$. The resulting simplicial
complex is a tree, denoted by $\mathfrak{T}_v$. 

There is an obvious action of $\op{GL}_2(K_v)$ on $\mathfrak{T}_v$ by
setting  
$$
\op{GL}_2(K_v)\times\mathfrak{T}_v\rightarrow\mathfrak{T}_v:(m,\Lambda)\mapsto
m\Lambda.
$$
\end{definition}

Note that $\op{GL}_2(K_v)$ acts transitively on the vertices of the
Bruhat-Tits tree for $v$, and $\op{SL}_2(K_v)$ acts transitively on the
vertices of fixed type. The center acts trivially, i.e., the above
actions on the building factor through actions of $\op{PGL}_2(K_v)$
and $\op{PSL}_2(K_v)$, respectively. 

For each vertex $x=[L]$ with representative  lattice $L$, there is a
bijection between the link $\operatorname{Lk}(x)$ and the lattices
$L'$ with $\pi_v L\subset L'\subset L$, hence with one-dimensional
subspaces of the two-dimensional $k_v$-vector space $L/\pi_v L$. Thus,
the elements of the link $\operatorname{Lk}(x)$ are in bijection with
the $k_v$-points of the projective line $\mathbb{P}^1(k_v)$,  where $k_v$
is the residue field of the valuation $v$ on $K$. 
In particular, the tree is homogeneous. 

We can be more precise about the correspondence between the link and
$\mathbb{P}^1(k_v)$: it is induced by mapping the points
$x\in\mathbb{A}^1(k_v)$ and $x=\infty$ to the lattice classes
$$
\left[L\cdot \left(\begin{array}{cc}
\pi_v&\widetilde{x}\\0&1\end{array}\right)\right]\textrm{ and
}
\left[L\cdot \left(\begin{array}{cc}
\pi_v^{-1}&0\\0&1\end{array}\right)\right], 
$$
where $\tilde{x}$ is a lift of $x\in k_v$ to
$\mathcal{O}_v$. This description of the correspondence of course
depends on a choice of basis vector for the lines in $k_v^2$, a choice
of lift to $L$ and a choice of completion to a basis of $L$.

Let $k$ be a field, let $\overline{C}$ be a smooth projective curve
over $k$, and set $C=\overline{C}\setminus\{P_1,\dots,P_s\}$ for a
non-empty finite set of pairwise distinct closed points
$P_1,\dots,P_s$ of $C$. Denote by $v_1,\dots,v_s$ the corresponding
valuations on the function field $K=k(C)$, and denote by
$\mathfrak{T}_i$ the Bruhat-Tits tree associated to the valuation
$v_i$. Then the group $\op{SL}_2(k[C])$ acts on
$\mathfrak{X}_C:=\mathfrak{T}_1\times\cdots\times\mathfrak{T}_s$ via
the embedding
$$
\op{SL}_2(k[C])\hookrightarrow \op{SL}_2(k(C))\hookrightarrow
\op{SL}_2(k(C)_{v_1})\times\cdots\times \op{SL}_2(k(C)_{v_s}).
$$
The product
$\mathfrak{X}_C=\mathfrak{T}_1\times\cdots\times\mathfrak{T}_s$ is the 
Bruhat-Tits building associated to $\op{SL}_2$ and the smooth affine
curve $C$. We view it as a cubical complex of dimension $s$ whose
non-degenerate cubes are products of edges from the factors
$\mathfrak{T}_i$. 

To describe the local structure of the product, we consider the link
of $0$-simplices. Recall that for $X$ a cubical complex and $\sigma$
a $0$-cube of $X$, the link  $\operatorname{Lk}_X(\sigma)$ of $\sigma$
in $X$ is the following simplicial complex: its $0$-simplices are the
$0$-cubes of $X$ connected to $\sigma$ via a $1$-cube,
and $0$-simplices $\sigma_0,\dots,\sigma_m$ span an $m$-simplex if
there exists an $(m+1)$-cube containing $\sigma,\sigma_0,\dots,\sigma_m$.  
For a vertex $(x_1,\dots,x_s)$ in
$\mathfrak{T}_1\times\cdots\times\mathfrak{T}_s$, the link of
$(x_1,\dots,x_s)$ is then the following simplicial complex: its set of
$0$-simplices is the disjoint union $\bigsqcup_{i=1}^s
\operatorname{Lk}_{\mathfrak{T}_i}(x_i)$, and for each choice of index
set $I\subseteq \{1,\dots,s\}$ of cardinality $n$ and elements
$\{y_i\in\operatorname{Lk}_{\mathfrak{T}_i}(x_i)\}_{i\in I}$, there is an 
$n$-simplex $(y_1,\dots,y_n)$ corresponding to the $(n+1)$-cube
spanned by $(x_1,\dots,x_s)$ and the $y_i$. In particular, for $n=s=2$,
the link of 
$(x_1,x_2)$ is the complete bipartite graph on the links of $x_1$ and
$x_2$. 

Again, we can be more precise about the lattices in the link: assume
that the vertex $(x_1,\dots,x_s)$ is represented by the lattice 
classes $([L_1],\dots,[L_s])$. Then for each choice of index $i\in
\{1,\dots,s\}$ and element $\alpha_i\in\mathbb{P}^1(k(P_i))$,
the corresponding point in the link of $(x_1,\dots,x_s)$ is given by
$([L_1],\dots,[L_i\cdot M_i],\dots,[L_s])$, where $M_i$ is the
corresponding matrix
$$
M_i= \left(\begin{array}{cc}
\pi_i&\widetilde{\alpha_i}\\0&1\end{array}\right)\textrm{ resp.
}
 \left(\begin{array}{cc}
\pi_i^{-1}&0\\0&1\end{array}\right).
$$

\subsection{Buildings and vector bundles}

As discussed, for a smooth affine curve
$C=\overline{C}\setminus\{P_1,\dots,P_s\}$, the groups
$\op{GL}_2(k[C])$ and $\op{SL}_2(k[C])$ act on the Bruhat-Tits
building $\mathfrak{X}_C$. It is well-known that the quotient can be
described in terms of vector bundles on the curve $\overline{C}$,
cf. \cite[Propositions II.2.4, II.2.5]{serre:book},  \cite{stuhler:76}
and \cite{stuhler:80}. We recall the necessary steps. 

\begin{definition}
Let $(L_1,\dots,L_s)$ be a tuple, in which $L_i$ is an
$\mathcal{O}_{P_i}$-lattice in $V=K^2$. 
We associate to this point in $\mathfrak{X}_C$
the unique coherent subsheaf $\mathcal{E}=\mathcal{E}(L_1,\dots,L_s)$
of the constant sheaf $V$ given by taking the stalk of $\mathcal{E}$
at a point $Q\in \overline{C}$ to be $L_i$ if $Q=P_i$ and
$\mathcal{O}_Q^2$ otherwise.  

Two vector bundles $\mathcal{E}_1$ and $\mathcal{E}_2$ are called
equivalent rel $\partial C$ if there exist integers $m_1,\dots,m_s$
such that 
$$
\mathcal{E}_1\otimes
\mathcal{O}_{\overline{C}}(-P_1)^{m_1}\otimes\cdots\otimes 
\mathcal{O}_{\overline{C}}(-P_s)^{m_s}\cong \mathcal{E}_2,
$$
where $\mathcal{O}_{\overline{C}}(-P)$ is the ideal sheaf of functions
vanishing at $P$.
\end{definition}

It is easy to see that $\mathcal{E}(L_1,\dots,L_s)$ is the sheaf of
sections of a rank two vector bundle over $\overline{C}$ whose
restriction to $C$ is trivial. In the following, there will be no
notational distinction between the sheaf of sections and the vector
bundle. It also follows easily that for scalars
$\lambda_1,\dots,\lambda_s\in K^\times$, we have  
$$
\mathcal{E}(L_1,\dots,L_s)\otimes
\mathcal{O}_{\overline{C}}(-P_1)^{v_1(\lambda_1)}\otimes\cdots\otimes
\mathcal{O}_{\overline{C}}(-P_s)^{v_s(\lambda_s)}
\cong 
\mathcal{E}(\lambda_1L_1,\dots,\lambda_sL_s).
$$

From these remarks, we have the following identification of the
quotient of the building in terms of vector bundles. 

\begin{proposition}
\label{prop:quotient}
The two assignments 
$$
(L_1,\dots,L_s)\mapsto ([L_1],\dots,[L_s])\in\mathfrak{X}_C,
\textrm{ and }
(L_1,\dots,L_s)\mapsto \mathcal{E}(L_1,\dots,L_s)
$$
induce a bijection between
\begin{enumerate}
\item the $0$-cells of the quotient
  $\op{SL}_2(k[C])\backslash\mathfrak{X}_C$, and 
\item equivalence (rel $\partial C$) classes of pairs
  $(\mathcal{E},f)$ where $\mathcal{E}$ is a rank two vector bundle on
  $\overline{C}$ whose restriction to $C$ is trivial, and
  $$f:\det\mathcal{E}\rightarrow
  \mathcal{O}_{\overline{C}}(-P_1)^{m_1}\otimes\cdots\otimes
  \mathcal{O}_{\overline{C}}(-P_s)^{m_s}$$ is a fixed isomorphism for
  suitable $m_i$.  
\end{enumerate}
This bijection further induces (by forgetting the fixed determinant) a 
bijection between 
\begin{enumerate}
\item the $0$-cells of the quotient
  $\op{GL}_2(k[C])\backslash\mathfrak{X}_C$, and 
\item equivalence (rel $\partial C$) classes of rank two vector
  bundles $\mathcal{E}$ on $\overline{C}$ whose restriction to $C$ is
  trivial. 
\end{enumerate}
\end{proposition}

\begin{remark}
The vector bundle classification is usually stated in different form:
the set of isomorphism classes of rank two vector bundles on
$\overline{C}$ whose restriction to $C$ is trivial is parametrized by
the double quotient 
$$
\op{GL}_2(k[C])\backslash \left(\prod_{i=1}^s
  \op{GL}_2(k(C)_{v_i})\right)/\left(\prod_{i=1}^s
  \op{GL}_2(\mathcal{O}_{v_i})\right). 
$$
In the building, we additionally divide out the centers of the groups
$\op{GL}_2(k(C)_{v_i})$ which leads to the additional equivalence (rel
$\partial C$) on vector bundles. There is no difference between using
$\op{GL}_2(k(C))$ or $\op{GL}_2(k(C)_{v_i})$ in the above, the right
cosets of the corresponding maximal compact group are the same.
\end{remark}

Under the correspondence described above, the link of a point
$x=([L_1],\dots,[L_s])$ corresponding to the vector bundle
$\mathcal{E}=\mathcal{E}(L_1,\dots,L_s)$ is given by the vector
bundles $\mathcal{E}'$ which arise from $\mathcal{E}$ by  elementary
transformations: there is an edge between the vector bundles
$\mathcal{E}$ and $\mathcal{E}'$ if there is an embedding
$\mathcal{E}'\hookrightarrow\mathcal{E}$ of $\mathcal{E}'$ as a
subsheaf of $\mathcal{E}$ and the quotient $\mathcal{E}/\mathcal{E}'$
is a torsion $\mathcal{O}_{\overline{C}}$-module of length one,
concentrated in one of the points $P_1,\dots,P_s$. Similarly,
the $n$-simplices in the link of $x$ are given by the choice of $n+1$
such elementary transformations at $n+1$ pairwise distinct points in
$\{P_1,\dots,P_s\}$. 

The correspondence between vertices in the quotient of the building
and equivalence classes of vector bundles also allows the
identification of stabilizers of vertices  with automorphism groups of
vector bundles.  

\begin{proposition}
\label{prop:aut}
Let $x=([L_1],\dots,[L_s])$ be a point in the building
$\mathfrak{X}_C$ with associated vector bundle
$\mathcal{E}=\mathcal{E}(L_1,\dots,L_s)$. Then 
we have the following description of stabilizers, where $\Stab(x;G)$
denotes the stabilizer of the vertex $x$ in the group $G$:
\begin{enumerate}
\item 
$\Stab(x;\op{GL}_2(k[C]))\cong \Aut(\mathcal{E})\times_{k^\times}
k[C]^\times$, where $k^\times\to\Aut(\mathcal{E})$ denotes the
embedding of the homotheties, 
\item $\Stab(x;\op{PGL}_2(k[C]))\cong \Aut(\mathcal{E})/k^\times$
\item $\Stab(x;\op{SL}_2(k[C]))\cong \Aut(\mathcal{E},f)$, where
  $\Aut(\mathcal{E},f)$ denotes those automorphisms commuting with the
  map   $f:\det\mathcal{E}\rightarrow
  \mathcal{O}_{\overline{C}}(-P_1)^{m_1}\otimes\cdots\otimes
  \mathcal{O}_{\overline{C}}(-P_s)^{m_s}$ from \prettyref{prop:quotient}.
\end{enumerate}
The stabilizers of a cube in any of the above groups are given by
the intersection of the respective stabilizers of its vertices. The
intersections are taken in the groups $\op{GL}_2(K)$, $\op{PGL}_2(K)$
and $\op{SL}_2(K)$, respectively.
\end{proposition}

\begin{proof}
(1) We first describe a homomorphism
$$\Aut(\mathcal{E})\times_{k^\times}k[C]^\times\rightarrow
\Stab(x;\op{GL}_2(k[C])).$$ 
Note that the bundle $\mathcal{E}$  comes with an explicit embedding
into $K^2$, in particular the restriction of $\mathcal{E}$ to
$C$ comes with a given trivialization. An automorphism
$\phi\in\Aut(\mathcal{E})$ can be restricted to a $k[C]$-linear 
automorphism of $\mathcal{E}|_C$, and the given trivialization allows
to write this as an element in $\op{GL}_2(k[C])$. This produces a
homomorphism $\Aut(\mathcal{E})\rightarrow \op{GL}_2(k[C])$ which
lands inside $\Stab(x;\op{GL}_2(k[C]))$ because the automorphism
preserves all the lattices. There is also a homomorphism 
$k[C]^\times\rightarrow \Stab(x;\op{GL}_2(k[C]))$  given by embedding
$k[C]^\times$ into the center of $\op{GL}_2(k[C])$ by 
$$
u\in k[C]^\times\mapsto\left(\begin{array}{cc}
u&0\\0&u\end{array}\right).
$$ 
This also preserves the lattice classes $[L_i]$, hence stabilizes
$x$. We can embed $k^\times\hookrightarrow \Aut(\mathcal{E})$ as the
homotheties, and $k^\times\hookrightarrow k[C]^\times$ as
constants. On $k^\times$, the homomorphisms given above agree, hence
we obtain the required homomorphism. This homomorphism is obviously
injective. 

By definition, the elements of $\Stab(x;\op{GL}_2(k[C]))$ preserve the
lattice classes $[L_i]$. Moreover, any element of $\op{GL}_2(k[C])$
preserves the standard lattice $\mathcal{O}_Q^2$ for
$Q\not\in\{P_1,\dots,P_s\}$. Therefore, any element of the stabilizer
which preserves the lattices $L_i$ (as opposed to just the lattice
classes) is in fact an automorphism of the vector bundle
$\mathcal{E}$. Now any element of the stabilizer can be factored as an
element which preserves the lattices $L_i$ and some central element,
hence the homomorphism is surjective.  

(2) is obtained by dividing out the center on both sides. On the
automorphism side, this is exactly dividing out the factor
$k[C]^\times$. On the stabilizer side, this is precisely the passage
from $\op{GL}_2$ to $\op{PGL}_2$. 

(3) is obtained by restricting to determinant $1$.  

Finally, the building is a CAT(0)-space, and $\op{GL}_2(k[C])$ acts
via isometries. Therefore, if an element stabilizes vertices, it also 
stabilizes the cube they span. Stabilizers of cubes can then be
computed by intersection of stabilizers of vertices inside the
respective group.
\end{proof}

\subsection{Recollection on Jacobians of curves}

We recall the Nagata exact sequence which describes units and Picard
groups of smooth affine curves. This will be needed for the
classification of (geometrically) split rank two bundles later on. 

\begin{lemma}
\label{lem:picc}
Let $\overline{C}$ be a smooth projective curve over a field $k$, 
let $P_1,\dots,P_s$ be closed points of degrees $d_i=[k(P_i):k]$, and
denote $C=\overline{C}\setminus\{P_1,\dots,P_s\}$. Then there is an exact
sequence  
$$
0\rightarrow \mathcal{O}(\overline{C})^\times=k^\times \rightarrow 
\mathcal{O}(C)^\times\rightarrow\bigoplus_{i=1}^s \mathbb{Z}
\stackrel{\phi}{\longrightarrow} \Pic(\overline{C})\rightarrow
\Pic(C)\rightarrow 0. 
$$
In particular:
\begin{enumerate}
\item
There is an exact sequence 
$$
0\rightarrow \Pic^0(C)\rightarrow
\Pic(C)\rightarrow\mathbb{Z}/\operatorname{gcd}(d_1,\dots,d_s)\mathbb{Z} 
\rightarrow 0,
$$
where
$\Pic^0(C)=\Pic^0(\overline{C})/(\operatorname{Im}\phi\cap
\Pic^0(\overline{C}))$.
\item
We have 
$$
k[C]^\times\cong k^\times\oplus\mathbb{Z}^{\dim\ker \phi}.
$$
\end{enumerate}
\end{lemma}

\begin{proof}
The exact sequence is a special case of a localization sequence in
K-theory or Chow groups. A more elementary proof of the exact sequence
may be found e.g. in \cite[Proposition~1]{rosen:sunits}. 

The morphism $\phi:\mathbb{Z}^s\to \Pic(\overline{C})$ sends the basis
vector $e_i$ to the line bundle $\mathcal{O}_{\overline{C}}(-P_i)$,
cf. \cite[Proposition II.6.5]{hartshorne}. Moreover,
the degree map provides the standard exact sequence
$$
0\rightarrow \Pic^0(\overline{C})\rightarrow
\Pic(\overline{C})\rightarrow\mathbb{Z}\rightarrow 0.
$$
The composition $\deg\circ \phi:\mathbb{Z}^s\rightarrow\Pic(\overline{C})
\stackrel{\deg}{\longrightarrow}\mathbb{Z}$ maps $e_i$ to $\deg(P_i)$,
so its image is generated by
$\operatorname{gcd}(d_1,\dots,d_s)$. Thus
$\ker(\deg\circ\phi)\cong\mathbb{Z}^{s-1}$ consists exactly of the degree zero
divisors, and the map $\ker\phi\rightarrow\Pic(\overline{C})$ factors
through $\mathbb{Z}^{s-1}\rightarrow\Pic^0(\overline{C})$. The rank of
the latter map varies, according to possible rational equivalence
relations between the points $P_1,\dots,P_s$. This proves (1). The
splitting and rank computation in (2) is obvious. 
\end{proof}

\section{Decomposable bundles and parabolic subcomplex}  
\label{sec:pardef}

In this section, we use the identification of \prettyref{prop:aut} to
discuss the structure of stabilizer subgroups of vertices in
$\mathfrak{X}_C$. Recalling some basics about automorphism groups of
rank two vector bundles, we see that the existence of a
non-trivial non-unipotent automorphism implies that the vector bundle
is geometrically split. Based on this observation, we consider the
subcomplex of the quotient which consists exactly of the cells which
correspond to geometrically decomposable bundles. The 
equivariant cohomology of this complex is a version of Farrell-Tate
cohomology, in fact agrees with it for finite base fields (and finite
coefficients away from the characteristic).  

\subsection{Automorphisms and decomposability of vector bundles}

We first recall the structure of automorphism groups of rank two
vector bundles. These results in particular imply that decomposability
of a rank two vector bundle can be seen from the structure of the
automorphism group.  

\begin{proposition}
\label{prop:auto}
Let $k$ be a field, let
$\overline{C}$ be a smooth projective curve over $k$, and let
$\mathcal{E}$ be a rank two vector bundle over $\overline{C}$.
\begin{enumerate}
\item If $\mathcal{E}$ is geometrically indecomposable, then we have
  $\End(\mathcal{E}_{\overline{k}})\cong \overline{k}\oplus
  \mathcal{N}il$. Every 
  automorphism is a product of a central element and a unipotent
  element.
\item If $\mathcal{E}$ is geometrically decomposable, then there
  exists a separable extension $L/k$ with $[L:k]=2$ such that
  $\mathcal{E}_L\cong \mathcal{L}\oplus\mathcal{L}'$ with
  $\mathcal{L}$ and $\mathcal{L}'$ line bundles over
  $\overline{C}\times_kL$. There are three possibilities: 
\begin{enumerate}
\item If $\mathcal{L}\cong\mathcal{L}'$, then both $\mathcal{L}$ and
  $\mathcal{L}'$ as well as the splitting are defined over $k$. We have
  $\End(\mathcal{E})\cong \op{M}_2(k)$ and
  $\Aut(\mathcal{E})\cong \op{GL}_2(k)$. 
\item If  $\mathcal{L}\not\cong\mathcal{L}'$ and the splitting is
  defined over $k$, then we have 
  (assuming $\deg \mathcal{L}\geq\deg\mathcal{L}'$) 
$$
\Aut(\mathcal{E})\cong\left\{\left(\begin{array}{cc}
a & c \\ 
0& b
\end{array}\right),\mid a,b\in k, ab\neq 0,\,
c\in \op{H}^0(\overline{C},\mathcal{L}\otimes(\mathcal{L}')^{-1})\right\}.
$$
\item If $\mathcal{L}\not\cong\mathcal{L}'$ and the splitting is not
  defined over $k$, then the automorphism group of $\mathcal{E}$ is
  the Weil restriction of a torus: 
$$
\Aut(\mathcal{E})\cong
\op{R}_{L/k}^1(\mathbb{G}_m). 
$$
\end{enumerate}
\end{enumerate}
\end{proposition}

\begin{proof}
(1) is due to Atiyah. In a complex analytic setting, it can be found as  
Proposition 15, its reformulation and Proposition 16 in
\cite{atiyah:connection}. The algebraic result is proved 
similarly using \cite{atiyah:ks}.

(2) can be found in \cite[p. 101]{serre:book}. 
\end{proof}

\begin{corollary}
A rank two vector bundle is geometrically split if and only if there
exists a non-central non-unipotent automorphism. 
\end{corollary}

\begin{remark}
A similar pattern appears for the  finite subgroups containing odd
order elements in arithmetic groups $\op{SL}_2(\mathcal{O}_{K,S})$:
case (a) is a dihedral group, case (b) is a diagonalizable finite
cyclic group, and case (c) a non-diagonalizable finite cyclic group.  
\end{remark}

\subsection{Parabolic subcomplex and an exact sequence}
In this paper, we are interested in the subcomplex of $\mathfrak{X}_C$
containing the decomposable bundles. More precisely, we consider the
subcomplex of $\mathfrak{X}_C$ containing exactly the cells which have
a non-central non-unipotent element in their stabilizer. This is
justified by the fact that the unipotent radicals 
of stabilizers will not be visible in homology with finite
coefficients away from the characteristic. 

\begin{lemma}
Let $k$ be a field, let $\overline{C}$ be a smooth projective curve
over $k$, and denote $C=\overline{C}\setminus\{P_1,\dots,P_s\}$. 
The subset of the building $\mathfrak{X}_C$ consisting of the cells
with non-unipotent stabilizer  is in fact a subcomplex. The action of
$\op{GL}_2(k[C])$ resp. $\op{SL}_2(k[C])$ on the building restricts to
an action on this subcomplex.  
\end{lemma}

\begin{proof}
As mentioned before, the stabilizers of $n$-cubes are the
intersections of the stabilizers of their vertices. In particular, if
an $n$-cube has non-unipotent stabilizer, then so do all its faces. 

If two cubes are conjugate by the $\op{GL}_2(k[C])$- or the
$\op{SL}_2(k[C])$-action, then so are their stabilizers. 
\end{proof}

\begin{remark}
Note that this subcomplex is not necessarily a full subcomplex. Even
if all the vertices of a cube have non-unipotent stabilizers, the cube
does not necessarily have a non-unipotent stabilizer. It frequently
happens that the automorphism groups of the vertices of a cube have
trivial intersection inside $\op{PGL}_2(K)$. 
\end{remark}

\begin{definition}
\label{def:pc}
\begin{itemize}
\item
The subcomplex of the building consisting of cells with non-unipotent
stabilizer in $\op{PGL}_2(k[C])$ is denoted by $\mathfrak{P}_C$ and is
called the \emph{parabolic subcomplex}. 
\item The quotient of the building $\mathfrak{X}_C$ modulo the subcomplex
  $\mathfrak{P}_C$ is denoted by $\mathfrak{U}_C$ and is called the
  \emph{unknown quotient}.
\end{itemize}
\end{definition}

\begin{remark}
The terminology ``parabolic subcomplex'' is supposed to underline that
the stabilizers in the subcomplex are strongly related to parabolic
(or better parahoric) subgroups of $\op{GL}_2(K)$ or $\op{SL}_2(K)$.  

The terminology ``unknown quotient'' is supposed to underline our
complete lack of knowledge of its structure, homology etc. There are
strong links to cuspidal phenomena in number theory: the quotient
$\mathfrak{U}_C$ contains the information away from infinity. In the
case of finite base fields, the unknown quotients contains most of the
compactly supported cohomology. However, ``cuspidal quotient'' seems
inappropriate terminology, overloaded as the term cuspidal is in
number theory and representation theory.

Exploration of the structure of the great unknown will be the subject
of further papers in the series.
\end{remark}

It now follows from the definition that we have an exact sequence of
chain complexes.

\begin{proposition}
\label{prop:split}
Let $k$ be a field, let $\overline{C}$ be a smooth projective curve
over $k$, and denote $C=\overline{C}\setminus\{P_1,\dots,P_s\}$. 
\begin{enumerate}
\item
Denote by $\Gamma$ one of the linear groups $\op{(P)GL}_2(k[C])$ or
$\op{(P)SL}_2(k[C])$.   
There is an exact sequence of $\Gamma$-chain complexes
$$
0\rightarrow \op{C}_\bullet(\mathfrak{P}_C)\rightarrow
\op{C}_\bullet(\mathfrak{X}_C) \rightarrow
\op{C}_\bullet(\mathfrak{U}_C)\rightarrow 0. 
$$
In particular, there is a long exact sequence of
$\Gamma$-equivariant homology groups:
$$
\cdots\rightarrow
\op{H}_{\bullet+1}^{\Gamma}(\op{C}_\bullet(\mathfrak{U}_C))\rightarrow
\op{H}_\bullet^{\Gamma}(\op{C}_\bullet(\mathfrak{P}_C))\rightarrow
\op{H}_\bullet^{\Gamma}(\op{C}_\bullet(\mathfrak{X}_C)) \rightarrow
\op{H}_\bullet^{\Gamma}(\op{C}_\bullet(\mathfrak{U}_C))\rightarrow\cdots 
$$
Since $\mathfrak{X}_C$ is contractible, we also have 
$\op{H}_\bullet^{\Gamma}(\op{C}_\bullet(\mathfrak{X}_C))\cong
\op{H}_\bullet(\Gamma)$. 
\item Assume now that $\Gamma$ is one of groups
  $\op{PGL}_2(k[C])$ or $\op{(P)SL}_2(k[C])$, and let $\ell$ be an odd
  prime different from the characteristic of $k$. For $i>s$, we have 
  $\op{H}_i^{\Gamma}(\op{C}_\bullet(\mathfrak{U}_C),\mathbb{Z}/\ell)=0$. In
  particular, we have  isomorphisms 
$$
\op{H}_i^{\Gamma}(\op{C}_\bullet(\mathfrak{P}_C))\cong
\op{H}_i(\Gamma,\mathbb{Z}/\ell). 
$$
\end{enumerate}
\end{proposition}

\begin{proof}
(1) is barely more than the definition of $\mathfrak{P}_C$ and
$\mathfrak{U}_C$. 

(2) By definition, the cells in $\mathfrak{U}_C$ have unipotent
stabilizer in $\op{PGL}_2(k[C])$ and $\op{PSL}_2(k[C])$, hence they
are uniquely $\ell$-divisible groups by assumption.  The same holds
for the stabilizers in $\op{SL}_2(k[C])$ since the center has order
$2$. 
In particular, the isotropy spectral sequence computing
$\op{H}^\Gamma_\bullet(\op{C}_\bullet(\mathfrak{U}_C),\mathbb{Z}/\ell)$ is
concentrated in the line $q=0$ since only the groups
$\op{H}_0(\Gamma_\sigma,\mathbb{Z}/\ell_\sigma)$ are non-trivial. The
spectral sequence therefore degenerates at the $E^2$-term and
converges to
$$
\op{H}_\bullet^\Gamma(\op{C}_\bullet(\mathfrak{U}_C),\mathbb{Z}/\ell)\cong
\op{H}_\bullet(\op{C}_\bullet(\Gamma\backslash\mathfrak{U}_C),\mathbb{Z}/\ell).
$$
The quotient only has cells in dimension $\leq s$, since the building
only has dimension $s$. This proves the claim.
\end{proof}

\begin{remark}
The exact sequence closely resembles the dual of the
  long exact sequence relating group cohomology to Farrell-Tate
  cohomology and the homology of the Steinberg module for groups of
  finite virtual cohomological dimension, 
  cf. \cite{brown:book}: 
$$
\cdots\rightarrow \widehat{H}^{\bullet-1}(\Gamma)\rightarrow
H_{n-\bullet}(\Gamma,\operatorname{St}^{\Gamma})\rightarrow
H^\bullet(\Gamma)\rightarrow 
\widehat{H}^\bullet(\Gamma)\rightarrow\cdots
$$
For finite base fields $k$, the groups $\op{GL}_2(k[C])$ and
$\op{SL}_2(k[C])$ have finite virtual $\ell$-cohomological dimension
if $\ell\neq\op{char}k$ and the long exact sequence of
\prettyref{prop:split} is the (dual of the) one for Farrell-Tate
cohomology. For infinite base fields, the groups $\op{GL}_2(k[C])$ and
$\op{SL}_2(k[C])$ do no longer have finite virtual cohomological
dimension. Nevertheless, we can see the homology of the parabolic
subcomplex as a replacement of Farrell-Tate homology in this setting. 
\end{remark}

\begin{definition}
Let $k$ be a field, let $\overline{C}$ be a smooth projective curve
over $k$, and denote $C=\overline{C}\setminus\{P_1,\dots,P_s\}$. 
Denote by $\Gamma$ one of the linear groups $\op{(P)GL}_2(k[C])$ or
$\op{(P)SL}_2(k[C])$.   
The $\Gamma$-equivariant homology of the parabolic subcomplex
$\mathfrak{P}_C$ is called the \emph{parabolic homology of $\Gamma$}. 
\end{definition}

\begin{remark}
There is a version of Tate cohomology for arbitrary groups, due to
Mislin \cite{mislin:tate}, which agrees with Farrell-Tate cohomology
for groups of finite virtual cohomological dimension. It is not clear
if Mislin's version of Tate cohomology agrees with the equivariant
cohomology of the parabolic subcomplex above. The difficulty is mainly
in the different definitions: while Mislin's version of Tate
cohomology is defined via killing projectives in the derived category
of $\mathbb{Z}[\Gamma]$-modules, the parabolic subcomplex has only a
very concrete geometric definition. 

A first step in comparing the two
cohomologies would be to compute the parabolic homology of a
projective $\mathbb{Z}[\Gamma]$-module which is possible using the
methods of the present paper. It could probably be shown that a
projective $\mathbb{Z}[\op{SL}_2(k[C])]$-module whose restriction to
the constant group rings $\mathbb{Z}[\op{SL}_2(k)]$ is injective, has
trivial parabolic homology. This would at least establish a morphism
from a ``relative'' Mislin-Tate cohomology to parabolic cohomology. 
\end{remark}

\subsection{Functoriality}
The parabolic subcomplex is also functorial with respect to morphisms
of curves: 

\begin{proposition}
\label{prop:functor1}
Let $f:\overline{D}\rightarrow \overline{C}$ be a finite morphism of
smooth projective curves over $k$, let $P_1,\dots,P_s$ be
points on $\overline{C}$, and let $Q_1,\dots,Q_t$ be points on
$\overline{D}$ not in the preimage of the $P_i$. 
Set $C=\overline{C}\setminus\{P_1,\dots,P_s\}$ and
$D=\overline{D}\setminus(\{f^{-1}(\{P_1,\dots,P_s\}) \cup
\{Q_1,\dots,Q_t\})$. 

(After possible subdivision of $\mathfrak{X}_C$) there is a natural
morphism $f^\ast:\mathfrak{X}_C\to\mathfrak{X}_D$ which induces 
a morphism $f^\ast:\mathfrak{P}_C\to\mathfrak{P}_D$ of parabolic
subcomplexes. The morphism $f^\ast:\mathfrak{P}_C\to\mathfrak{P}_D$
is equivariant with respect to the natural group homomorphism
$\op{GL}_2(k[C])\to \op{GL}_2(k[D])$. This assignment is functorial.

There are induced morphisms
$f^\ast:\op{GL}_2(k[C])\backslash\mathfrak{X}_C\to
\op{GL}_2(k[D])\backslash\mathfrak{X}_D$ (and similarly for
$\op{SL}_2$ in place of $\op{GL}_2$ resp. $\mathfrak{P}$ in place of
$\mathfrak{X}$). In terms of the vector bundle interpretation of the
quotient, these morphisms are identified with pullback of vector
bundles. 
\end{proposition}

\begin{proof}
The morphism of curves induces a morphism of function fields
$f^\ast:k(C)\to k(D)$. This is compatible with the valuations and
hence induces a morphism
$$
\prod_{i=1}^s\op{GL}_2(k(C))/(\prod_{i=1}^s\op{GL}_2(\mathcal{O}_{P_i})\cdot
k(C)^\times) \to 
\prod_{i=1}^{\tilde{s}+t}\op{GL}_2(k(D))/
(\prod_{i=1}^{\tilde{s}+t}\op{GL}_2(\mathcal{O}_{Q_i})\cdot k(D)^\times),
$$
where the factor $\op{GL}_2(k(C))$ for the point $P_i$ is mapped
diagonally to the factors for $f^{-1}(P_i)$. 
This morphism is obviously functorial. It is also obviously
equivariant with respect to $\op{GL}_2(k[C])\to\op{GL}_2(k[D])$, both
groups embedded diagonally into $\op{GL}_2(k(C))^s$ and
$\op{GL}_2(k(D))^{\tilde{s}+t}$, respectively. 

The cubical structure is slightly more complicated to take care of: a 
matrix for an elementary transformation 
$$
\left(\begin{array}{cc}
\pi_i^{a_i}&\alpha_i\\0&1
\end{array}\right)\textrm{ is mapped to }
\left(\begin{array}{cc}
f^\ast(\pi_i)^{a_i}&f^\ast(\alpha_i)\\0&1
\end{array}\right)^g,
$$
and so maps a $1$-cube to the diagonal of a $g$-cube, where $g$ is
cardinality of $f^{-1}(P_i)$. After suitably subdividing, the above
map on vertices extends to a morphism of cell complexes. 
This proves all the claims for the morphism
$f^\ast:\mathfrak{X}_C\to\mathfrak{X}_D$. For the parabolic
subcomplexes, it then suffices to note that if a cell
$\sigma\in\mathfrak{X}_C$ has a non-unipotent non-central subgroup in
the stabilizer, then the same is true for its image: any non-unipotent
non-central element in the stabilizer in $\op{GL}_2(k[C])$ will have
non-unipotent non-central image in $\op{GL}_2(k[D])$. Therefore
$f^\ast$ restricts to the parabolic subcomplexes, and all the
auxiliary properties are inherited.
\end{proof}

\section{Local structure of parabolic subcomplexes}
\label{sec:local}

In this section, we work out the local structure of the parabolic
subcomplex $\mathfrak{P}_C$ resp. its quotient $\Gamma\backslash
\mathfrak{P}_C$, for $\Gamma$ any of the linear groups
$\op{(P)GL}_2(k[C])$ or $\op{(P)SL}_2(k[C])$. The general procedure is
as follows: for a geometrically split bundle, we completely describe
the action of its automorphism group on the link of the corresponding
point in the building. The link $\op{Lk}_{\mathfrak{P}}(x)$ of the
point $x$ in $\mathfrak{P}_C$ is given by those simplices in
$\op{Lk}_{\mathfrak{X}}(x)$ which are  fixed by a non-unipotent
non-central element of the stabilizer of $x$. The link of the image of
$x$ in the quotient $\Gamma\backslash\mathfrak{P}_C$  is then given by
the quotient
$\op{Stab}(x;\Gamma)\backslash\op{Lk}_{\mathfrak{P}}(x)$. These local
computations will be used in the next section to describe the global
structure of the quotient $\Gamma\backslash\mathfrak{P}_C$.

\subsection{Action of the automorphism group on the links}

We first investigate the action of automorphism groups of rank two
bundles on the links of their corresponding points in
$\mathfrak{X}_C$. Let $\mathcal{E}$ be a vector bundle corresponding
to the vertex $x$ in $\mathfrak{X}_C$. By \prettyref{prop:aut}, the
stabilizer of $x$ is essentially determined by the automorphism group
of the vector bundle $\mathcal{E}$. Combining this with 
\prettyref{prop:auto}, we see that $x$ is in $\mathfrak{P}_C$ if and
only if $\mathcal{E}$ is geometrically split. Moreover,
\prettyref{prop:auto} provides a detailed description of the
isomorphism types of possible stabilizer groups.  

To obtain a description of the action of those stabilizer groups on
the links, recall from \prettyref{sec:buildings} that for $x$ a vertex
of the building, the corresponding link is a disjoint union
$\bigsqcup_{i=1}^s\mathbb{P}^1(k(P_i))$ with the ``complete
$s$-partite simplicial structure''. The point $x$ is given by
specifying $\mathcal{O}_{P_i}$-lattices in $V=K^2$, or equivalently
(using the standard basis of $V$) by specifying matrices $M_i\in
\op{GL}_2(K_{v_i})$, for $i\in\{1,\dots,s\}$. As in
\prettyref{sec:buildings}, the link of $x$ can be described in terms
of these data: the vertices in the $i$-th component
$\mathbb{P}^1(k(P_i))$ are points corresponding to the tuples 
$$
\left(M_1,\dots,M_i\cdot\left(\begin{array}{cc} 
\pi_i&\tilde{\alpha_i}\\0&1
\end{array}\right),\dots,M_s\right), \tilde{\alpha_i}\in
\mathcal{O}_{P_i},\textrm{ and}$$
$$
\left(M_1,\dots,M_i\cdot \left(\begin{array}{cc}
\pi_i^{-1}&0\\0&1
\end{array}\right),\dots,M_s\right).
$$
Note that the point in the building only depends on the $\prod 
\op{GL}_2(\mathcal{O}_{P_i})\cdot K^\times_{v_i}$-right coset of the above,
in particular, the first line above only depends on the residue class
$\alpha_i$ of $\tilde{\alpha_i}$ in $\mathbb{A}^1(k(P_i))$. 
The stabilizer of the point $(M_1,\cdots,M_s)$ in the group
$\Gamma=\op{(P)GL}_2(k[C])$ or $\Gamma=\op{(P)SL}_2(k[C])$ is then
exactly the subgroup of $\Gamma$ leaving invariant the $\prod
\op{GL}_2(\mathcal{O}_{P_i})\cdot K_{v_i}^\times$-right coset of
$(M_1,\dots,M_s)$. The next proposition describes the action of the
stabilizer on the link by computing explicitly the action of the
stabilizer on the right cosets of matrices as given above.

\begin{lemma}
\label{lem:help1}
Let $k$ be a field, let $\overline{C}$ be a smooth projective curve
over $k$ and set $C=\overline{C}\setminus\{P_1,\dots,P_s\}$. Let
$\mathcal{E}\cong\mathcal{L}\oplus\mathcal{L}'$ be a split vector
bundle over $\overline{C}$, and denote by $x$ the point of
$\mathfrak{X}_C$ corresponding to $\mathcal{E}$. Denote by $\Gamma(C)$
one of the linear groups $\op{(P)GL}_2(k[C])$ or
$\op{(P)SL}_2(k[C])$. 
\begin{enumerate}
\item 
There exist closed points $Q_1,\dots,Q_t$ such that the restrictions
of both $\mathcal{L}$ and $\mathcal{L}'$ to $C'=\overline{C}\setminus
\{P_1,\dots,P_s,Q_1,\dots,Q_t\}$ are trivial. 
\item There is an inclusion of
buildings $\mathfrak{X}_C\hookrightarrow \mathfrak{X}_{C'}$ which is
equivariant for the natural homomorphism $\Gamma(C)\to
\Gamma(C')$ and  induces an isomorphism on stabilizers
$\op{Stab}(x;\op{P}\Gamma(C))\cong \op{Stab}(x;\op{P}\Gamma(C'))$.  
\item
As a consequence of (2), the inclusion
$\op{Lk}_{\mathfrak{X}_C}(x)\hookrightarrow
\op{Lk}_{\mathfrak{X}_{C'}}(x)$ is equivariant for the action of
$\op{Stab}(x;\Gamma(C))$. 
\end{enumerate}
\end{lemma}

\begin{proof}
(1) First note that $\mathcal{L}^{-1}|_C\cong\mathcal{L}'|_C$ because
$\mathcal{E}|_C$ is trivial. Because $\overline{C}$ is a smooth curve,
the usual identification of divisor class group and Picard group
allows to write $\mathcal{L}\cong \mathcal{O}_C(\sum_{i=1}^s n_i
P_i+\sum_{j=1}^t m_j Q_j)$. The restriction of $\mathcal{L}$ to 
$C'=\overline{C}\setminus\{P_1,\dots,P_s,Q_1,\dots,Q_t\}$ is then
obviously trivial, cf. also \prettyref{lem:picc}.

For (2), we can consider the quotient description of the vertices of the
building. The induced map is induced by inclusion of the first $s$
factors: 
$$
\prod_{i=1}^s\op{GL}_2(K_{v_i})/(\op{GL}_2(\mathcal{O}_{P_i})\cdot
K^\times_{v_i}) \to 
$$
$$
\to\prod_{i=1}^s\op{GL}_2(K_{v_i})/(\op{GL}_2(\mathcal{O}_{P_i})\cdot
K^\times_{v_i})\times 
\prod_{j=1}^t\op{GL}_2(K_{v_j})/(\op{GL}_2(\mathcal{O}_{Q_j})\cdot
K^\times_{v_j})
$$ 
The map is evidently injective on vertices. Since the building is a
cubical complex, i.e., every cube is uniquely determined by its
vertices, the map is an inclusion of complexes. As the groups
$\Gamma(C)$ and $\Gamma(C')$ act via diagonal inclusion into the
product $\prod \op{GL}_2(K_{v_i})$, the equivariance of the inclusion
of buildings is also clear. Similarly, it is clear that the map
restricts to a homomorphism of stabilizers. The stabilizers have been
described in \prettyref{prop:aut}, and from that description it is
clear that the map induced on stabilizers has to be an isomorphism -
the relevant stabilizer groups are directly related to automorphism
groups of vector bundles over the curve $\overline{C}$ and independent
of the choice of affine subcurve. 

(3) is a direct consequence of (2). 
\end{proof}

\begin{remark}
Note that the isomorphism on stabilizer groups in (2) above only holds
for the projective (special) linear groups. In the case of
$\op{GL}_2$, the stabilizer group can become bigger,
cf. \prettyref{prop:aut}. However, the $k[C]^\times$-part of the stabilizer
acts trivially on the link anyway, so this does not affect application
of the above lemma to $\op{GL}_2(k[C])$. 
\end{remark}

\begin{proposition}
\label{prop:action}
Let $k$ be a field, let $\overline{C}$ be a smooth projective curve
over $k$ and set $C=\overline{C}\setminus\{P_1,\dots,P_s\}$. Let
$\mathcal{E}\cong\mathcal{L}\oplus\mathcal{L}'$ be a split vector
bundle over $\overline{C}$. Denote by $x$ the point of
$\mathfrak{X}_C$ corresponding to $\mathcal{E}$. 

There exist closed points $Q_1,\dots,Q_t$ of $C$ such that the image
of $x$ in the double quotient
$$
\op{GL}_2(k[C'])\backslash
\prod_{i=1}^{s}\op{GL}_2(K_{v_i})/(\op{GL}_2(\mathcal{O}_{P_i})\cdot 
  K^\times_{v_i})\times 
\prod_{j=1}^{t}\op{GL}_2(K_{v_j})/(\op{GL}_2(\mathcal{O}_{Q_j})\cdot
  K^\times_{v_j})
$$
has a representative of the following form:
$$
\left(\left(\begin{array}{cc}
\pi_1^{a_1}&0\\0&1\end{array}\right),\dots,
\left(\begin{array}{cc}
\pi_t^{a_t}&0\\0&1\end{array}\right)\right), \quad\sum a_i\geq 0.
$$

\begin{enumerate}[(a)]
\item Assume $\mathcal{L}\cong\mathcal{L}'$, i.e., the stabilizer is
  conjugate to a diagonally embedded $\op{GL}_2(k)$. The action of
  $\op{GL}_2(k)$ on the $i$-th component of the link is the   standard
  action of   $\op{GL}_2$ on $\mathbb{P}^1(k(P_i))$, hence the action of
  $\op{GL}_2(k)$ on the link is the diagonal standard action. 

\item Assume $\mathcal{L}\not\cong\mathcal{L}'$, i.e., the stabilizer
  is conjugate to the group 
$$
\op{B}_2(a_1,\dots,a_t)=\left\{
\left(\begin{array}{cc}
a & 0\\ f&b
\end{array}\right)\mid a,b\in k, ab\neq 0, v_i(f)\geq -a_i
\right\}.
$$
There are two cases for the action of the stabilizer on the $i$-th
component of the link of $x$: if $v_i(f)=-a_i$, then the action is the
``standard'' action. If $v_i(f)>-a_i$, the action is trivial. 
\end{enumerate}
\end{proposition}

\begin{proof}
We first prove the statement about the double coset
representative. Denoting
$C'=\overline{C}\setminus\{P_1,\dots,P_s,Q_1,\dots,Q_t\}$ as in
\prettyref{lem:help1}, the double quotient written classifies exactly
equivalence  (rel $\partial C'$) classes of rank two vector bundles on
$\overline{C}$ whose restriction to $C'$ is trivial. Therefore, the
elements corresponding to $\mathcal{E}\cong
\mathcal{L}\oplus\mathcal{L}'$ and
$\mathcal{L}\otimes(\mathcal{L}')^{-1}\oplus\mathcal{O}$ are equal in
the double quotient. 
It is easy to see that the double coset representative given
corresponds to the vector bundle 
$$
\mathcal{O}\left(-\sum_{i=1}^sa_iP_i-\sum_{j=1}^ta_{s+j}Q_j\right)
\oplus\mathcal{O} 
$$
Therefore, the statement follows if the $a_i$ are chosen such that
$$
\mathcal{O}\left(-\sum_{i=1}^sa_iP_i-\sum_{j=1}^ta_{s+j}Q_j\right)\cong
\mathcal{L}\otimes(\mathcal{L}')^{-1}.
$$
Note that it is necessary to enlarge the set $\{P_1,\dots,P_s\}$
(using \prettyref{lem:help1})
because $\mathcal{L}\otimes(\mathcal{L}')^{-1}$ might not be trivial
over $C$, whence it would not correspond to an element in the double
quotient $k[C]^\times\backslash \prod
K^\times_{v_i}/\prod\mathcal{O}^\times_{P_i}$. 

Next, we explain how to reduce the statements (a) and (b) to
statements about the given double coset representative. By
\prettyref{lem:help1} (3), the inclusion of links
$\op{Lk}_{\mathfrak{X}_C}(x)\to \op{Lk}_{\mathfrak{X}_{C'}}(x)$ is
equivariant for the action of the stabilizer of $x$. To determine this
action, it thus suffices to determine the action of
$\op{Stab}(x;\op{GL}_2(k[C]))$ on the vertices of
$\mathfrak{X}_{C'}$. Moreover, if two points $x$ and $x'$ of
$\mathfrak{X}_{C'}$ are $\op{GL}_2(k[C'])$-conjugate, then so are
their stabilizers as well as the actions of the stabilizers on the
respective links. We can therefore replace $x$ by any other
representative of the vertices of
$\op{GL}_2(k[C'])\backslash\mathfrak{X}_{C'}$. The vertices of the
latter quotient are exactly the double cosets in the
statement. Summing up, the link $\op{Lk}_{\mathfrak{X}_C}(x)$ together
with the action of the stabilizer of $x$ can be equivariantly
embedded into the link of the chosen representative
above in $\mathfrak{X}_{C'}$. 

Now we come to the prove of part (a). From the above, we have that the
stabilizer of $x$ is $\op{GL}_2(k[C'])$-conjugate to the stabilizer of
the double coset $(1,\dots,1)$ where all entries are identity
matrices. The stabilizer of this double coset inside
$\op{GL}_2(k[C'])$ is obviously $\op{GL}_2(k)$, embedded diagonally as
constants. Up to right multiplication with elements from
$\op{GL}_2(\mathcal{O}_{P_i})\cdot K_{v_i}^\times$, we find
$$
\left(\begin{array}{cc}
a&b\\c&d
\end{array}\right)\cdot
\left(\begin{array}{cc}
\pi_i^{-1}&0\\0&1
\end{array}\right)
=
\left(\begin{array}{cc}
a\pi_i^{-1}&b\\c\pi_i^{-1}&d
\end{array}\right)
\sim
\left(\begin{array}{cc}
\pi_i&\frac{a}{c}\\0&1
\end{array}\right), c\neq 0
$$
$$
\left(\begin{array}{cc}
a&b\\c&d
\end{array}\right)\cdot
\left(\begin{array}{cc}
\pi_i&\tilde{\alpha}\\0&1
\end{array}\right)
=
\left(\begin{array}{cc}
a\pi_i&a\tilde{\alpha}+b\\c\pi_i&c\tilde{\alpha}+d
\end{array}\right)
\sim
\left(\begin{array}{cc}
\pi_i&\frac{a\tilde{\alpha}+b}{c\tilde{\alpha}+d}\\0&1
\end{array}\right), c\tilde{\alpha}+d\neq 0
$$
and the missing cases are dealt with similarly.
But this is indeed the standard action of $\op{GL}_2(k)$ on
$\mathbb{P}^1(k(P_i))$. 

(b) By what was said above, it suffices to compute the stabilizer and
its action on the link in the case of the diagonalized double coset
representative. 

Up to right multiplication by elements of
$\op{GL}_2(\mathcal{O}_{P_i})\cdot K_{v_i}^\times$, we have the
following 
$$
\left(\begin{array}{cc}
a & 0\\ f&b
\end{array}\right)\cdot
\left(\begin{array}{cc}
\pi_i^{a_i}&0\\0&1\end{array}\right)=
\left(\begin{array}{cc}
a\pi_i^{a_i} & 0 \\
f\pi_i^{a_i} & b
\end{array}\right)
\sim
\left(\begin{array}{cc}
\pi_i^{a_i} & 0 \\
0 & 1
\end{array}\right),
$$
whenever $v_i(f)\geq -a_i$. From $\sum a_i\geq 0$, it follows that the
matrices above are the only ones possibly stabilizing the
representative. This shows that the stabilizer is indeed as described. 

We now study the action of the stabilizer $\op{B}_2(a_1,\dots,a_t)$
on the link of $x$. For the point $\infty$, we have
$$
\left(\begin{array}{cc}
a&0\\f&b
\end{array}\right)
\left(\begin{array}{cc}
\pi_i^{a_i-1}&0\\0&1
\end{array}\right)
=
\left(\begin{array}{cc}
a\pi_i^{a_i-1}&0\\f\pi_i^{a_i-1}&b
\end{array}\right)
$$
and there are two cases to consider: if $v_i(f)=-a_i$, then
$v_i(f\pi_i^{a_i-1})=-1$, and we can rewrite the matrix using right
multiplication with $\op{GL}_2(\mathcal{O}_{P_i})\cdot K_{v_i}^\times$
to the form 
$$
\left(\begin{array}{cc}
\pi_i^{a_i+1}&g\\0&1
\end{array}\right),
$$
where $v(g)=a_i$ and the leading term is $\frac{a}{f_0}$ with
$f_0$ the leading term of $f$ (as in the constant case). 
On the other
hand, if $v_i(f)>-a_i$, then $v_i(f\pi_i^{a_i-1})\geq 0$, and right
multiplication allows to rewrite this to
$\operatorname{diag}(\pi_i^{a_i-1},1)$. In this case, the action is
trivial. For any other point
$\alpha\in\mathbb{P}^1(k(\mathbb{P}^1))\setminus\{\infty\}$, we
consider    
$$
\left(\begin{array}{cc}
1&0\\f&1
\end{array}\right)
\left(\begin{array}{cc}
\pi_i^{a_i+1}&\pi_i^{a_i}\tilde{\alpha}\\0&1
\end{array}\right)
=
\left(\begin{array}{cc}
\pi_i^{a_i+1}&\pi_i^{a_i}\tilde{\alpha}\\
f\pi_i^{a_i+1}&f\pi_i^{a_i}\tilde{\alpha}+1
\end{array}\right).
$$
The restriction $a=b=1$ is only for simplification of
exposition. In the case $v_i(f)=-a_i$, we can use right multiplication
to transform the matrix to 
$$
\left(\begin{array}{cc}
\pi_i^{a_i+1}&\pi_i^{a_i}\cdot\frac{\tilde{\alpha}}{1+f\pi_i^{a_i}\tilde{\alpha}}\\
0&1
\end{array}\right).
$$
We can use further right multiplication to transform
$\frac{\tilde{\alpha}}{1+f\pi_i^{a_i}\tilde{\alpha}}$ to $\frac{\tilde{\alpha}}{1+c\tilde{\alpha}}$ with $c$ the constant
term of $f\pi_i^{a_i}$. If $v_i(f)>-a_i$, then this constant term is
$0$, and the action is trivial.

For $v_i(f)=-a_i=0$, this is the standard action of the  Borel
$\mathbb{B}_2(k)$ on $\mathbb{P}^1(k(P_i))$. For $v_i(f)=-a_i\neq 0$ ,
the action is still the same, only that $f$ acts via its constant
coefficient. If $v_i(f)> -a_i$, then the action is trivial. The
claim is proved. 
\end{proof}

\begin{remark}
A similar result holds for bundles which are indecomposable but
geometrically split. Let $\mathcal{E}$ be such a bundle, with
corresponding point $x\in\mathfrak{X}_C$, and let $L/k$ be a splitting
field for the bundle $\mathcal{E}$. Then the morphism of curves
$\phi:\overline{C}\times_kL\to\overline{C}$ induces a morphism of
buildings $\mathfrak{X}_C\hookrightarrow \mathfrak{X}_{C\times_kL}$
which embeds $\mathfrak{X}_C$ as the $\op{Gal}(L/k)$-invariants into
$\mathfrak{X}_{C\times_kL}$. The results of \prettyref{prop:action}
above then imply that the action of the stabilizer of $x$ on
$\op{Lk}_{\mathfrak{X}_C}(x)$ is also the standard one, coming from
the embedding of $\op{Stab}(x;\op{GL}_2(k[C]))$ as
$\op{Gal}(L/k)$-invariants of $\op{Stab}(x;\op{GL}_2(L[C]))$. 
For this and other reasons, our main results are restricted to the
case of algebraically closed base fields. 
\end{remark}

\subsection{Local structure} 
Having determined the action of vertex stabilizers on links, we are
now ready to describe the local structure of the quotients
$\Gamma\backslash\mathfrak{P}_C$ as subcomplexes of
the respective quotients of $\Gamma\backslash\mathfrak{X}_C$, where as
usual $\Gamma$ is one of the linear groups $\op{(P)GL}_2(k[C])$ or
$\op{(P)SL}_2(k[C])$. For a point $x\in 
\mathfrak{X}_C$, the link $\op{Lk}_{\Gamma\backslash\mathfrak{P}_C}$
is the quotient
$\op{Stab}(x;\Gamma)\backslash\op{Lk}_{\mathfrak{P}_C}(x)$. The
following proposition now uses \prettyref{prop:action} to determine
$\op{Lk}_{\mathfrak{P}_C}(x)$ as well as its quotient modulo
$\op{Stab}(x;\Gamma)$.  

\begin{proposition}
\label{prop:local}
Let $k$ be an algebraically closed field, let $\overline{C}$ be a
smooth projective curve over $k$ and denote
$C=\overline{C}\setminus\{P_1,\dots,P_s\}$. Let
$\mathcal{E}\cong\mathcal{L}\oplus\mathcal{L}'$ be a split rank two
vector bundle, and let $x$ be the corresponding point in the building
$\mathfrak{X}_C$. Denote by $\Gamma$ one of the linear groups
$\op{(P)GL}_2(k[C])$ or $\op{(P)SL}_2(k[C])$.  

Denote by $\mathfrak{A}_x$ the intersection of an apartment
$\mathfrak{A}$ with $\op{Lk}_{\mathfrak{X}_C}(x)$.
\begin{enumerate}[(a)]
\item If $\mathcal{L}'\cong\mathcal{L}^{-1}$, then
  $\operatorname{Lk}_{\Gamma\backslash\mathfrak{P}}(x)$ is isomorphic
  to the quotient of $\mathfrak{A}_x$ modulo the antipodal
  $\mathbb{Z}/2$-action. The star
  $\operatorname{St}_{\Gamma\backslash\mathfrak{P}}(x)$ is isomorphic
  to a cone over (a cubical complex of the homotopy type of)
  $\mathbb{RP}^{s-1}$. The stabilizer of $x$ is 
  $\Gamma(k)$, everything else is stabilized (up to a unipotent group)
  by a maximal torus of $\Gamma$. 
\item If $\mathcal{L}'\not\cong\mathcal{L}^{-1}$,
  then $\operatorname{Lk}_{\Gamma\backslash\mathfrak{P}}(x)$ is
  isomorphic to $\mathfrak{A}_x$. The stabilizer is (up to a
  unipotent group) a maximal torus in $\Gamma(k)$. 
\end{enumerate}
\end{proposition}

\begin{proof}
Recall that for a field $F$, the set $\mathbb{P}^1(F)$ is a building
of type $A_1$. The group $\op{GL}_2(F)$ acts in the standard way on
$\mathbb{P}^1(F)$. The standard apartment is the subset
$\{0,\infty\}\subseteq\mathbb{P}^1(F)$, its setwise stabilizer in
$\op{GL}_2(F)$ is the normalizer $\op{N}_2(F)$ of the standard maximal torus
$\op{T}_2(F)$, its pointwise stabilizer is the standard maximal torus
$\op{T}_2(F)$.  

Now we consider the  ``complete $n$-partite simplicial
complex'' $\op{Lk}_{\mathfrak{X}_C}(x)$ on
the disjoint union $\bigsqcup_{i=1}^s\mathbb{P}^1(k)$, with the
diagonal action of 
$\op{GL}_2(k)$. Note that $\mathfrak{A}_x$ is conjugate to the full
subcomplex of $\op{Lk}_{\mathfrak{X}_C}(x)$ spanned by
$\bigsqcup_{i=1}^s\{0_i,\infty_i\}\subseteq\bigsqcup_{i=1}^s\mathbb{P}^1(k)$. It
is easy to see that the setwise stabilizer of $\mathfrak{A}_x$ is
conjugate to the normalizer  $\op{N}_2(k)$ and the pointwise
stabilizer is conjugate to the standard maximal torus
$\op{T}_2(k)$. Note that $\mathfrak{A}_x$ is a 
triangulation of the $(s-1)$-sphere. Moreover, the Weyl group
$\op{N}_2(k)/\op{T}_2(k)\cong\mathbb{Z}/2$ acts on $\mathfrak{A}_x$ by
the antipodal action.  

(a) By \prettyref{prop:action}, we have
$\Stab(x;\Gamma)=\Gamma(k)$ acting via the diagonal standard
action on the link $\bigsqcup_{i=1}^s\mathbb{P}^1(k)$.  A simplex
$\sigma\in\operatorname{Lk}_{\mathfrak{X}_C}(x)$ is contained in
$\operatorname{Lk}_{\mathfrak{P}_C}(x)$ if its stabilizer in
$\op{GL}_2(k)$ contains a non-central non-unipotent
element. The latter is equivalent to $\sigma$ being stabilized by some
maximal torus of $\op{GL}_2(k)$. But this is the case precisely if the
corresponding simplex is $\op{GL}_2(k)$-conjugate to one in
$\mathfrak{A}_x$. We see that every simplex in
$\op{Lk}_{\mathfrak{P}_C}(x)$ is $\op{GL}_2(k)$-conjugate to one in
$\mathfrak{A}_x$. As remarked above, the setwise  stabilizer of the
standard apartment is the normalizer of the maximal 
torus corresponding to $\mathfrak{A}_x$, the pointwise stabilizer
of the standard apartment is the maximal torus itself, and the Weyl
group acts via the antipodal action. This proves the claim about the
link. The claim about the star follows, because the star is the cone
over the link. The maximal torus stabilizes $\mathfrak{A}_x$
pointwise, so each simplex in
$\op{Lk}_{\Gamma\backslash\mathfrak{P}_C}(x)$ is stabilized at least
by the maximal torus. As there are no simplices stabilized by
$\op{GL}_2(k)$, the claim follows from the knowledge of stabilizers in
\prettyref{prop:auto}. 

(b) By \prettyref{prop:action}, we have
$\Stab(x;\op{GL}_2(k[C]))\cong k^n\rtimes \op{T}_2(k)$. The maximal
torus acts via the diagonal standard
action on the link $\bigsqcup_{i=1}^s\mathbb{P}^1(k)$. Fixing a
section of $k^n\rtimes\op{T}_2(k)\to\op{T}_2(k)$, there 
are exactly two fixed points. Any two such choices  of sections yield
actions which are conjugate (by the unipotent radical) to the
action with fixed points exactly $0$ and $\infty$ in 
each disjoint summand $\mathbb{P}^1(k)$. Then also the simplices
connecting these points are stabilized by the maximal
torus. Furthermore, any simplex which is stabilized by some
non-unipotent element is stabilized by some such maximal torus. 
Summing up, we see that every simplex in $\op{Lk}_{\mathfrak{P}_C}(x)$
is $\op{Stab}(x;\op{GL}_2(k[C]))$-conjugate to a simplex in
$\mathfrak{A}_x$, and $\mathfrak{A}_x$ is a strict fundamental domain
for the $\op{Stab}(x;\op{GL}_2(k[C]))$-action on
$\op{Lk}_{\mathfrak{P}_C}(x)$. This proves the claim. 

The same statements also are true for $\op{SL}_2(k[C])$ instead of
$\op{GL}_2(k[C])$ used above.
\end{proof}

\begin{remark}
The link $\operatorname{Lk}_{\mathfrak{P}_C}(x)$ is the preimage of
these links under the projection
$\mathfrak{P}_C\rightarrow\Gamma\backslash\mathfrak{P}_C$. However, it is
more difficult to describe explicitly due to unipotent subgroups in
the stabilizer of $x$. We ignored these in the local description of
$\Gamma\backslash\mathfrak{P}_C$ because they will not be visible in
the homology computations anyway. 
\end{remark}

\begin{remark}
The reason why we restricted to algebraically closed base fields in
\prettyref{prop:local} is the following: although we have identified
in \prettyref{prop:action} that the action of $\op{GL}_2(k)$ on
$\mathbb{P}^1(k(P_i))$ is the standard one, this does not allow to
give an easy description of the quotient
$\op{Stab}(x;\Gamma)\backslash\op{Lk}_{\mathfrak{P}}(x)$. 

The phenomenon is already visible in the case where $s=1$ and 
$[k(P_1):k]=2$. The action of $\op{GL}_2(k)$ on $\mathbb{P}^1(k(P_1))$
is no longer transitive but has three orbits, one orbit isomorphic to
$\mathbb{P}^1(k)$, any of the two other orbits is a
symmetric space generalizing the upper 
half plane $\mathbb{H}$. The stabilizer of points in $\mathbb{P}^1(k)$
is a Borel subgroup, the stabilizer of points in the other orbits is a
rank one non-split torus. 

In general, depending on the size of $k(P_i)$ and the number of degree
$2$ subfields, there can be many relevant non-split tori in
$\op{GL}_2(k)$ fixing points in $\mathbb{P}^1(k(P_i))$. The
description of the quotient
$\op{Stab}(x;\Gamma)\backslash\op{Lk}_{\mathfrak{P}}(x)$ then strongly
depends on the extension $k(P_i)/k$. In any case, the quotient tends
to  be much more complicated than in the algebraically closed case
handled in \prettyref{prop:local}. 
\end{remark}

\section{Global structure of parabolic subcomplexes}
\label{sec:global}

After having determined the local structure of the parabolic
subcomplex $\Gamma\backslash\mathfrak{P}_C$, we will now turn to the
investigation of the global structure. The main results state that
$\Gamma\backslash\mathfrak{P}_C$ is a disjoint union of subcomplexes,
indexed by classes of line bundles over $C$, and the structure of the 
subcomplex associated to the line bundle $\mathcal{L}$ can be
explicitly described: if $\mathcal{L}\not\cong\mathcal{L}^{-1}$ then
it is a torus related to the classifying space of a maximal torus in
$\op{GL}_2(K)$, and if $\mathcal{L}\cong\mathcal{L}^{-1}$ it is a
slightly more complicated space related to the classifying space of
the normalizer of a maximal torus in $\op{GL}_2(K)$. 

The topology of the quotient $\op{GL}_2(k[C])\backslash\mathfrak{X}_C$ - and
in particular the homotopy type at infinity - was studied already by
Stuhler in \cite{stuhler:76} and \cite{stuhler:80} in the case where
$k=\mathbb{F}_q$. The results below can be seen as a version of his
work: our results work over algebraically closed base fields $k$ and
describe a part of $\op{GL}_2(k[C])\backslash\mathfrak{X}_C$ slightly
larger than a neighbourhood of infinity. The method of proof here is
still the one already employed by Stuhler: we will write down model
complexes, and the local computations from \prettyref{sec:local} will
imply that the model complexes are isomorphic to the components of
$\Gamma\backslash\mathfrak{P}_C$. 

\subsection{Line bundles and connected components}
\begin{definition}
Let $k$ be a field, let $\overline{C}$ be a smooth projective curve
and denote $C=\overline{C}\setminus\{P_1,\dots,P_s\}$. 
We define $\mathcal{K}(C)$ to be the quotient of the Picard group
$\op{Pic}(C)$ modulo the involution $\iota:\mathcal{L}\mapsto
\mathcal{L}^{-1}$. 
\end{definition}

As the notation suggests, the set  $\mathcal{K}(C)$ is both related to
the Kummer variety associated to the curve $\overline{C}$ as well as
the set of komponents of $\Gamma\backslash\mathfrak{P}_C$:

\begin{lemma}
\label{lem:conn}
Let $k$ be an algebraically closed field, let $\overline{C}$ be a
smooth projective curve and denote
$C=\overline{C}\setminus\{P_1,\dots,P_s\}$. Denote by $\Gamma$ one of
the groups $\op{(P)GL}_2(k[C])$ or $\op{(P)SL}_2(k[C])$. 
Then we have a bijection 
$$
\pi_0(\Gamma\backslash\mathfrak{P}_C)\cong \mathcal{K}(C).
$$
\end{lemma}

\begin{proof}
The set $\pi_0$ is defined by taking the set of vertices of 
$\Gamma\backslash\mathfrak{P}_C$ and dividing by the equivalence
relation generated by $1$-cells of $\Gamma\backslash\mathfrak{P}_C$. 

We first write down a map $(\Gamma\backslash\mathfrak{P}_C)_0\to
\mathcal{K}(C)$: by \prettyref{def:pc}, \prettyref{prop:auto} and our
assumption that $k$ is algebraically closed, any vertex $x$ of
$\Gamma\backslash\mathfrak{P}_C$ corresponds to a split vector bundle
$\mathcal{E}\cong\mathcal{L}\oplus\mathcal{L}'$. 
Recall from \prettyref{prop:quotient} that
$$
\det\mathcal{E}\cong\mathcal{L}\otimes\mathcal{L}'\cong
\mathcal{O}(a_1P_1)\otimes\cdots\otimes \mathcal{O}(a_sP_s)
$$
for suitable $a_i\in\mathbb{Z}$. Using \prettyref{lem:picc}, we see
that the assignment 
$$
(\Gamma\backslash\mathfrak{P}_C)_0\to\mathcal{K}(C):
x=[\mathcal{L}\oplus\mathcal{L}']\mapsto [\mathcal{L}]
$$
is well-defined. It does not depend on the choice of (rel $\partial
C$)-equivalence class of a representative $\mathcal{E}$ of $x$ and does
not depend on the splitting. 

The map is clearly surjective: for any element $l\in\mathcal{K}(C)$,
we can find a line bundle $\mathcal{L}$ on $\overline{C}$ such that
$l=\{\mathcal{L}|_C,\mathcal{L}^{-1}|_C\}$. Then
$\mathcal{L}\oplus\mathcal{L}^{-1}$ is a vector bundle on
$\overline{C}$ with trivial determinant, hence corresponds to a point
$x\in\Gamma\backslash\mathfrak{P}_C$ which is mapped to $l$ by the map
described above. 

It suffices to show injectivity. Let $x_1$ and $x_2$ be two points in
$\Gamma\backslash\mathfrak{P}_C$ corresponding to vector bundles
$\mathcal{E}_1$ and $\mathcal{E}_2$ which are both mapped to
$l\in\mathcal{K}(C)$. This means that we have splittings 
$$
\mathcal{E}_1\cong\mathcal{L}_1\oplus\mathcal{L}_1', \qquad 
\mathcal{E}_2\cong\mathcal{L}_2\oplus\mathcal{L}_2',
$$
with $\mathcal{L}_1|_C\cong(\mathcal{L}_2|_C)^{\pm 1}$. Up to
equivalence (rel $\partial C$) and switching summands, we can assume
$\mathcal{L}_1\cong\mathcal{L}_2$. But then we have
$$
\mathcal{L}_1'\otimes(\mathcal{L}_2')^{-1}\cong
\mathcal{O}(a_1P_1)\otimes\cdots\otimes\mathcal{O}(a_sP_s). 
$$
We need to find a chain of $1$-cubes in
$\Gamma\backslash\mathfrak{P}_C$ connecting $\mathcal{E}_1$ to 
$\mathcal{E}_2$. It is easy to find a chain of elementary
transformations connecting
$\mathcal{O}(a_1P_1)\otimes\cdots\otimes\mathcal{O}(a_sP_s)$ to the
trivial line bundle, noting that
$\mathcal{O}(-P)\hookrightarrow\mathcal{O}$ is an inclusion of line
bundles whose quotient is a torsion sheaf of length one on the curve
$\overline{C}$.  This provides a chain of $1$-cubes in
$\Gamma\backslash\mathfrak{X}_C$ connecting $\mathcal{E}_1$ and
$\mathcal{E}_2$. 

We claim that this chain is contained in $\mathfrak{P}_C$. Let
$\mathcal{L}$ be any line bundle and consider the two bundles
$\mathcal{L}\oplus\mathcal{L}^{-1}\otimes \mathcal{O}(P)$ and
$\mathcal{L}\oplus\mathcal{L}^{-1}$. We can choose a double coset
representative of $\mathcal{L}\oplus\mathcal{L}^{-1}$ which only
contains diagonal matrices. In particular, the stabilizer contains the
diagonal maximal torus. The elementary transformation from
$\mathcal{L}\oplus\mathcal{L}^{-1}$ to
$\mathcal{L}\oplus\mathcal{L}^{-1}\otimes\mathcal{O}(P)$ is then given
by the diagonal matrix $\op{diag}(1,\pi)$ where $\pi$ is a uniformizer
at $P$. In particular, this elementary transformation is compatible
with the action of the diagonal torus. In other words, identifying the
link of $\mathcal{L}\oplus\mathcal{L}^{-1}$ with $\mathbb{P}^1$ with
the standard action of the diagonal torus, the above elementary
transformation corresponds to the point $\infty$ which is fixed by the
diagonal torus. Therefore, the $1$-cube corresponding to the above
elementary transformation actually lies in the subcomplex
$\Gamma\backslash\mathfrak{P}_C$. The general assertion is proved in
just the same way, only more notational effort is required to keep
track of all the coefficients $a_i$. 
Summing up, the path constructed between $\mathcal{E}_1$ and
$\mathcal{E}_2$ lies inside $\Gamma\backslash\mathfrak{P}_C$, and we
have thus shown that the map
$\pi_0(\Gamma\backslash\mathfrak{P}_C)\to\mathcal{K}(C)$ is
injective. 
\end{proof}

\begin{remark}
For a number field $K$, one considers the action of 
$\op{SL}_2(\mathcal{O}_K)$ on its corresponding symmetric space
$\mathfrak{X}$. In this situation, there is a well-known bijection
between the cusps of the locally symmetric space
$\op{SL}_2(\mathcal{O}_K)\backslash\mathfrak{X}$ and the class group
of $K$, going back (at least) to Siegel. The above computation of the
connected components of $\Gamma\backslash\mathfrak{X}_C$ is a
variation on this. The obvious change is that the class group has been
replaced by the Picard group of the curve $C$. A non-obvious change is
the appearance of the involution $\iota$. This comes from the fact
that for each line bundle $\mathcal{L}$ with
$\mathcal{L}\not\cong\mathcal{L}^{-1}$, there are two components of
the homotopy type at infinity, belonging to $\mathcal{L}$ and
$\mathcal{L}^{-1}$, which belong to the same connected component of
$\mathfrak{P}_C$. 
\end{remark}

\subsection{Model complexes of groups}
Now we come to the description of the connected components of
$\Gamma\backslash\mathfrak{P}_C$. There are two cases, depending on
the line bundle $\mathcal{L}$ indexing the component: if
$\mathcal{L}\not\cong\mathcal{L}^{-1}$ then the component is related
to a ``maximal torus of $\Gamma$'', and if
$\mathcal{L}\cong\mathcal{L}^{-1}$ then the component is related to
the ``normalizer of a maximal torus of $\Gamma$''. 
The results are proved by first defining  models for the corresponding
quotients and then provide isomorphisms to the respective connected
components.  

\begin{definition}
\label{def:models}
Let $\mathbb{Z}^s\subseteq\mathbb{R}^s$ be the standard
lattice. Denote by $\mathcal{Z}^s$ the cubical complex generated by
it. We define the  groups 
$$
\mathcal{T}=\left\{(a_1,\dots,a_s)\in\mathbb{Z}^s\mid
  \sum_{i=1}^s a_i[P_i]=0 \textrm{ in
  }\Pic(\overline{C})\right\}, \quad\textrm{ and }\quad
\mathcal{ST}=\mathcal{T}^2.
$$
The groups $\mathcal{T}$ and $\mathcal{ST}$ act on $\mathbb{R}^s$ by
translations and preserve the standard lattice
$\mathbb{Z}^s$. Therefore, the translation actions of $\mathcal{T}$ and
$\mathcal{ST}$ induces respective actions on the cubical complex
$\mathcal{Z}^s$. 

We define the groups $\mathcal{N}$ (resp. $\mathcal{SN}$) to be the 
subgroups of the euclidean group $\op{E}(s)$ generated by $\mathcal{T}$
(resp. $\mathcal{ST}$) and all point inversions in the
points $(a_1,\dots,a_s)$ for which $\sum_{i=1}^s a_i[P_i]=0$ in
$\op{Pic}(\overline{C})$.  
\end{definition}

\begin{remark}
Note that the complex $\mathcal{Z}^s/\mathcal{T}$ already appears in
\cite{stuhler:80} and is used there for the description of the
homotopy type of $\op{PGL}_2(k[C])\backslash\mathfrak{X}_C$ at
infinity. 
\end{remark} 

We note that the group $\mathcal{T}$ is a free abelian group whose
rank equals the rank of the kernel of $\phi$ in
\prettyref{lem:picc}, and in particular is isomorphic to the group of
non-constant units in $k[C]$. There are many reasons for calling this
group $\mathcal{T}$ - it is realized as a subgroup of translations, it
is a subgroup of a maximal torus of $\op{GL}_2(K)$, and the quotient
of $\mathcal{Z}^s/\mathcal{T}$ has the homotopy type of a torus.

The next results will describe the structure of the connected
components of $\Gamma\backslash\mathfrak{P}_C$ as quotients of
apartments modulo suitable torus- or normalizer-like groups.

\begin{lemma}
\label{lem:apartment}
Let $k$ be an algebraically closed field, let $\overline{C}$ be a
smooth projective curve and denote
$C=\overline{C}\setminus\{P_1,\dots,P_s\}$. Let $\mathcal{L}$ be a
line bundle on $\overline{C}$. There exists an apartment
$\mathfrak{A}_{\mathcal{L}}$ of $\mathfrak{X}_C$ and a morphism of
cubical complexes
$\phi_{\mathcal{L}}:\mathcal{Z}^s\to\mathfrak{A}_{\mathcal{L}}$ such
that   
\begin{enumerate}
\item $\mathfrak{A}_{\mathcal{L}}$ is contained in $\mathfrak{P}_C$, 
\item the $0$-cells of the image of $\mathfrak{A}_{\mathcal{L}}$ in
  $\op{GL}_2(k[C])\backslash\mathfrak{X}_C$ correspond to equivalence
  (rel $\partial C$) classes of rank two bundles of the form
$$
\left(\mathcal{L}\otimes\mathcal{O}(a_1P_1)\otimes
\cdots\otimes\mathcal{O}(a_sP_s)\right)\oplus\mathcal{L}^{-1}, \quad
a_i\in\mathbb{Z},
$$
\item $\phi_{\mathcal{L}}$ is an isomorphism.
\end{enumerate}
\end{lemma}

\begin{proof}
Recall the double quotient description of the $0$-cells of
$\op{GL}_2(k[C])\backslash\mathfrak{X}_C$:
$$
\op{GL}_2(k[C])\backslash\prod_{i=1}^s\op{GL}_2(K)/
\prod_{i=1}^s(\op{GL}_2(\mathcal{O}_{P_i})\cdot K^\times). 
$$
Let $(M_1,\dots,M_s)\in\op{GL}_2(K)^s$ be a representative of the
double coset corresponding to the bundle
$\mathcal{L}\oplus\mathcal{L}^{-1}$. Denoting by $\op{T}_2(K)$ the
diagonal maximal torus of $\op{GL}_2(K)$, we consider the maximal
torus 
$$
\op{T}(\mathcal{L})=(M_1\op{T}_2(K)M_1^{-1},\dots,
M_s\op{T}_2(K)M_s^{-1})\subset \op{GL}_2(K)^s.  
$$ 
We define $\mathfrak{A}_{\mathcal{L}}$ to be the full cubical
subcomplex of $\mathfrak{X}_C$ whose $0$-cubes are the
$\op{T}(\mathcal{L})$-orbit of $(M_1,\dots,M_s)$.  By construction
this is an apartment of $\mathfrak{X}_C$, corresponding to a choice of
frame in each of the $s$ copies of $K^2$. 

(1) We see that the subgroup 
$(M_1\op{T}_2(k)M_1^{-1},\dots,
M_s\op{T}_2(k)M_s^{-1})\subset \op{T}(\mathcal{L})$
preserves lattice classes in $\mathfrak{A}_{\mathcal{L}}$ because
$\op{T}_2(k)\subseteq \op{GL}_2(\mathcal{O}_{P_i})$. In particular,
this subgroup stabilizes $\mathfrak{A}_{\mathcal{L}}$
\emph{pointwise}.  
On the other hand,
because $(M_1,\dots,M_s)$ represents a split rank two bundle
$\mathcal{L}\oplus\mathcal{L}^{-1}$, the point $(M_1,\dots,M_s)$ is
stabilized by a maximal torus of $\op{GL}_2(k)$. From
\prettyref{prop:action},  we know that the action of the stabilizer of
$\mathcal{L}\oplus\mathcal{L}^{-1}$  on the link is such that the
stabilizer contains a maximal torus which fixes pointwise the
intersection of $\mathfrak{A}_{\mathcal{L}}$ with the link
$\op{Lk}_{\mathfrak{X}_C}(\mathcal{L}\oplus\mathcal{L}^{-1})$. In particular,
all the vertices in the link
$\op{Lk}_{\mathfrak{A}_{\mathcal{L}}}(\mathcal{L}\oplus\mathcal{L}^{-1})$
have the same stabilizer as
$\mathcal{L}\oplus\mathcal{L}^{-1}$. An inductive argument shows that
this is true for all vertices in $\mathfrak{A}_{\mathcal{L}}$. This
implies that the intersection  
$$
(M_1\op{T}_2(k)M_1^{-1},\dots, M_s\op{T}_2(k)M_s^{-1})\cap
\op{GL}_2(k[C])
$$
contains a maximal torus of $\op{GL}_2(k)$. Then all cubes of
$\mathfrak{A}_{\mathcal{L}}$ are stabilized by this, and we see that
$\mathfrak{A}_{\mathcal{L}}\subset \mathfrak{P}_C$. 

(2) is easy to see: up to
$\prod_{i=1}^s(\op{GL}_2(\mathcal{O}_{P_i})\cdot K^\times)$, the
elements of $\op{T}_2^s$ are of the form
$$
A=\left(\left(\begin{array}{cc}
\pi_1^{a_1}&0\\0&1\end{array}\right),\cdots,
\left(\begin{array}{cc}
\pi_s^{a_s}&0\\0&1\end{array}\right)\right).
$$
But then the elements of $\op{T}(\mathcal{L})\cdot(M_1,\dots,M_s)$ are
those of the form
$$\left(M_1\cdot\left(\begin{array}{cc}
\pi_1^{a_1}&0\\0&1\end{array}\right),\cdots,
M_s\cdot\left(\begin{array}{cc}
\pi_s^{a_s}&0\\0&1\end{array}\right)\right),
$$
and these are exactly representatives for bundles of the form
$$
\left(\mathcal{L}\otimes\mathcal{O}(-a_1P_1)\otimes
\cdots\otimes\mathcal{O}(-a_sP_s)\right)\oplus\mathcal{L}^{-1}, \quad
a_i\in\mathbb{Z}. 
$$

Now  the morphism
$\phi_{\mathcal{L}}:\mathcal{Z}^s\to\mathfrak{A}_{\mathcal{L}}$ is
defined on $0$-cubes as 
$$
(a_1,\dots,a_s)\mapsto 
\left(M_1\cdot\left(\begin{array}{cc}
\pi_1^{a_1}&0\\0&1\end{array}\right),\cdots,
M_s\cdot\left(\begin{array}{cc}
\pi_s^{a_s}&0\\0&1\end{array}\right)\right),
$$
and we have seen above that this assignment indeed takes vertices to
$0$-cubes in $\mathfrak{A}_{\mathcal{L}}$. It is also easy to see that
if vertices span a cube in $\mathcal{Z}^s$, then their
$\phi_{\mathcal{L}}$-images span a cube in
$\mathfrak{A}_{\mathcal{L}}$. Therefore, the above definition of
$\phi_{\mathcal{L}}$ indeed provides a morphism of cubical complexes.
(3) follows because $\mathfrak{A}_{\mathcal{L}}$ is a Coxeter complex
for the affine Weyl group $\mathbb{Z}^s\rtimes\mathbb{Z}/2^s$, and
hence the morphism $\phi_{\mathcal{L}}$ has an obvious inverse mapping 
$$
\left(M_1\cdot\left(\begin{array}{cc}
\pi_1^{a_1}&0\\0&1\end{array}\right),\cdots,
M_s\cdot\left(\begin{array}{cc}
\pi_s^{a_s}&0\\0&1\end{array}\right)\right)\mapsto (a_1,\dots,a_s).
$$
\end{proof}

\begin{proposition}
\label{prop:torus}
Let $k$ be an algebraically closed field, let $\overline{C}$ be a
smooth projective curve and denote
$C=\overline{C}\setminus\{P_1,\dots,P_s\}$. Let $\mathcal{L}$ be a
line bundle on $\overline{C}$ such that
$\mathcal{L}|_C\not\cong\mathcal{L}^{-1}|_C$.  
\begin{enumerate}
\item 
The connected component of
$\op{GL}_2(k[C])\backslash\mathfrak{P}_C$ containing the vector bundle
$\mathcal{L}\oplus\mathcal{L}^{-1}$ is isomorphic to
$\mathcal{Z}^s/\mathcal{T}$. 
The stabilizers of cells are of the form
$(k^n\rtimes (k^\times)^2)\times_{k^\times}k[C]^\times$ for some $n$
depending on the cell, i.e., up 
to the unipotent radical all the cells have stabilizer
$\op{T}\times_{k^\times} k[C]^\times$ where $\op{T}$ is a maximal
torus in $\op{GL}_2(k)$.  
\item 
The connected component of 
$\op{SL}_2(k[C])\backslash\mathfrak{P}_C$ containing the vector bundle
$\mathcal{L}\oplus\mathcal{L}^{-1}$ is isomorphic to
$\mathcal{Z}^s/\mathcal{ST}$.
The stabilizers of cells are of the form 
$k^n\rtimes k^\times$ for some $n$ depending on the cell, i.e., up to
the unipotent radical all the cells have stabilizer a maximal torus in
$\op{SL}_2(k)$.  
\end{enumerate}
In both cases,  the quotient has the homotopy type of a torus whose 
dimension equals  the rank of $k[C]^\times/k^\times$. 
\end{proposition}

\begin{proof}
We give the argument for $\op{SL}_2$, the argument for $\op{GL}_2$ is
similar. The proof is structured as follows: in step (1), we describe
a homomorphism $\tau:\mathcal{ST}\to \op{SL}_2(k[C])$ and show in step
(2) that the map $\phi_{\mathcal{L}}$ from \prettyref{lem:apartment}
is $\rho$-equivariant. Steps (3) to (5) deduce the claim from these
facts.  

(1) Recall from \prettyref{def:models} that $\mathcal{ST}$ was defined
as the subgroup of $\mathbb{Z}^s$ containing the tuples
$(2a_1,\dots,2a_s)$ with  $\sum_{i=1}^sa_i[P_i]=0$ in
$\op{Pic}(\overline{C})$.  This is a free abelian group isomorphic to 
$k[C]^\times/k^\times$: the condition on the sum in the Picard group
implies that $\sum a_i[P_i]$ is a principal divisor $(f)$, and $f\in
k[C]^\times/k^\times$. Choose an embedding $u:\mathcal{ST}\to
k[C]^\times$ which induces one such isomorphism $\mathcal{ST}\cong
k[C]^\times/k^\times$ and define a homomorphism $\tau:\mathcal{ST}\to
\op{SL}_2(K)^s$ by sending $(2a_1,\dots,2a_s)$ to  
$$
(M_1,\dots,M_s)\cdot\left(\begin{array}{cc}
u(2a_1,\dots,2a_s)&0\\0&u(2a_1,\dots,2a_s)^{-1}
\end{array}\right)\cdot(M_1^{-1},\dots,M_s^{-1}).
$$
By definition, $\tau(2a_1,\dots,2a_s)$ stabilizes the image of
$(M_1,\dots,M_s)$ in $\op{SL}_2(k[C])\backslash\mathfrak{P}_C$  if 
$\sum_{i=1}^sa_i[P_i]=0$ in $\op{Pic}(\overline{C})$: it sends the
vector bundle $\mathcal{L}\oplus\mathcal{L}^{-1}$ to
$(\mathcal{L}\otimes 
\mathcal{O}(-2a_1P_1)\otimes\cdots\otimes\mathcal{O}(-2a_sP_s))
\oplus\mathcal{L}^{-1}$, but these are both isomorphic under the above
condition. In this case, $\tau(2a_1,\dots,2a_s)\in\op{SL}_2(k[C])$
because it stabilizes the double coset $(M_1,\dots,M_s)$ in
$\op{SL}_2(k[C])\backslash \prod
\op{GL}_2(K)/(\prod\op{GL}_2(\mathcal{O}_{P_i})\cdot K^\times)$. We
see that $\tau:\mathcal{ST}\to\op{SL}_2(K)^s$ factors through a
homomorphism $\mathcal{ST}\to\op{SL}_2(k[C])$ which we will still
denote $\tau$. 

(2) We show that the embedding $\phi_{\mathcal{L}}$ from
\prettyref{lem:apartment} is $\tau$-equivariant. The element
$(2a_1,\dots,2a_s)$ acts on the model complex $\mathcal{Z}^s$ by the
evident translation. The element $\tau(2a_1,\dots,2a_2)$ sends 
$(M_1,\dots,M_s)\in\mathfrak{A}_{\mathcal{L}}$ to
$$
\left(M_1\cdot\left(\begin{array}{cc}
\pi_1^{a_1} & 0\\0&\pi_1^{-a_1}\end{array}\right),\cdots,
M_s\cdot\left(\begin{array}{cc}
\pi_s^{a_s} & 0\\0&\pi_s^{-a_s}\end{array}\right)\right)
$$
which is also the translation by $(2a_1,\dots,2a_s)$ on
$\mathfrak{A}_{\mathcal{L}}$. 

(3) From (1) and (2) above we obtain a map of quotient cell complexes
$$
\overline{\phi_{\mathcal{L}}}:\mathcal{Z}^s/\mathcal{ST}\to
\op{SL}_2(k[C])\backslash\mathfrak{P}_C.
$$
We first show that this map is surjective onto the
$\mathcal{L}$-component of $\op{SL}_2(k[C])\backslash\mathfrak{P}_C$. 
By \prettyref{lem:apartment} (2), we know that the $0$-cells of the
image are exactly the equivalence (rel $\partial C$) classes of rank
two bundles of the form
$(\mathcal{L}\otimes\mathcal{O}(a_1P_1)\otimes\cdots\otimes
\mathcal{O}(a_sP_s))\oplus\mathcal{L}^{-1}$. For each
$(a_1,\dots,a_s)\in \mathcal{Z}^s$, the map
$\overline{\phi_{\mathcal{L}}}$ induces an isomorphism from the star
of $(a_1,\dots,a_s)$ in $\mathcal{Z}^s$ to the star of
$\overline{\phi_{\mathcal{L}}}$ in
$\op{SL}_2(k[C])\backslash\mathfrak{P}_C$, by
\prettyref{prop:local}. Therefore, the map
$\overline{\phi_{\mathcal{L}}}$ is a surjection onto the
$\mathcal{L}$-component of $\op{SL}_2(k[C])\backslash\mathfrak{P}_C$. 

(4) We show that the map $\overline{\phi_{\mathcal{L}}}$ is
injective. First, we consider $0$-cells. Let $(a_1,\dots,a_s)$ and
$(b_1,\dots,b_s)$ be two vertices of $\mathcal{Z}^s$. The vector
bundles associated to these two vertices are 
$$
(\mathcal{L}\otimes\mathcal{O}(a_1P_1)\otimes\cdots\otimes
\mathcal{O}(a_sP_s))\oplus\mathcal{L}^{-1} \textrm{ and }
(\mathcal{L}\otimes\mathcal{O}(b_1P_1)\otimes\cdots\otimes
\mathcal{O}(b_sP_s))\oplus\mathcal{L}^{-1}.
$$
These vector bundles are isomorphic precisely when
$\sum(a_i-b_i)P_i=0$ in $\op{Pic}(\overline{C})$. But then the two
original vertices are $\mathcal{ST}$-conjugate. This shows injectivity
on $0$-cells. Injectivity for $n$-cells now follows from injectivity
of $0$-cells and the fact - established in \prettyref{prop:local}  and
used in step (3) above - that $\overline{\phi_{\mathcal{L}}}$ induces
isomorphisms on stars. 

(5) Finally, the stabilizer statements follow from
\prettyref{prop:aut} and the fact that all vector bundles are split
with non-isomorphic direct summands. It is also clear that the
quotient modulo the group of translations has the homotopy type of a
torus of the rank claimed.
\end{proof}

\begin{proposition}
\label{prop:normalizer}
Let $k$ be an algebraically closed field, let $\overline{C}$ be a
smooth projective curve and denote
$C=\overline{C}\setminus\{P_1,\dots,P_s\}$. Let $\mathcal{L}$ be a
line bundle on $\overline{C}$ such that
$\mathcal{L}|_C\cong\mathcal{L}^{-1}|_C$.   
\begin{enumerate}
\item 
The connected component of
$\op{GL}_2(k[C])\backslash\mathfrak{P}_C$ containing the vector bundle
$\mathcal{L}\oplus\mathcal{L}$ is isomorphic to
$\mathcal{Z}^s/\mathcal{N}$. 
The image of $(0,\dots,0)$ is stabilized by $\op{GL}_2(k)\times_{k^\times}
k[C]^\times$. 
The stabilizers of all other cells are of the form
$(k^n\rtimes (k^\times)^2)\times_{k^\times} k[C]^\times$ for some $n$
depending on the cell, i.e., up to the unipotent radical all the other
cells have stabilizer $\op{T}\times_{k^\times} k[C]^\times$ where $\op{T}$ is a
maximal torus in $\op{GL}_2(k)$.  
\item
 The connected component of
$\op{SL}_2(k[C])\backslash\mathfrak{P}_C$ containing the vector bundle
$\mathcal{L}\oplus\mathcal{L}$ is isomorphic to
$\mathcal{Z}^s/\mathcal{SN}$. 
There are $2^{\operatorname{rk}\mathcal{T}}$ points with stabilizer
$\op{SL}_2(k)$, they are the $\mathcal{ST}$-conjugacy classes of
$\mathcal{T}$-translates of $(0,\dots,0)$. 
The stabilizers of all other cells are of the form 
$k^n\rtimes k^\times$ for some $n$ depending on the cell, i.e., up to
the unipotent radical all the other cells have stabilizer a maximal
torus in $\op{SL}_2(k)$.  
\end{enumerate}
\end{proposition}

\begin{proof}
We give the argument for $\op{SL}_2$, the argument for $\op{GL}_2$ is
similar but easier. The proof structure is similar to the proof of
\prettyref{prop:torus}. 

(1) We construct a homomorphism $\nu:\mathcal{SN}\to\op{SL}_2(k[C])$
as follows: on the subgroup $\mathcal{ST}$, the morphism $\nu$ agrees
with $\tau$. Now for $\sum_{i=1}^s a_i[P_i]=0$ in
$\op{Pic}(\overline{C})$ and $r(a_1,\dots,a_s)$ the point-reflection
with center $(a_1,\dots,a_s)$, the map $\nu$ is defined to be   
$$
(M_1,\dots,M_s)\cdot\left(\begin{array}{cc}
0 & u(a_1,\dots,a_s)\\-u(a_1,\dots,a_s)^{-1}&0
\end{array}\right)\cdot(M_1^{-1},\dots,M_s^{-1}).
$$
Then $\nu(r(a_1,\dots,a_s))$  again stabilizes the image of
$$
\left(M_1\cdot\left(\begin{array}{cc}
1&0\\0&\pi_1^{a_1}\end{array}\right),\dots,M_s\cdot
\left(\begin{array}{cc}
1&0\\0&\pi_s^{a_s}\end{array}\right)\right)
$$
in $\op{SL}_2(k[C])\backslash\mathfrak{P}_C$ under
the condition $\sum_{i=1}^s a_i[P_i]=0$. We see that
$\nu:\mathcal{SN}\to\op{SL}_2(K)^s$ factors through a 
homomorphism $\mathcal{SN}\to\op{SL}_2(k[C])$ which we will still
denote $\nu$. 

(2) We show that the embedding $\phi_{\mathcal{L}}$ from
\prettyref{lem:apartment} is $\nu$-equivariant. Equivariance on the
subgroup $\mathcal{ST}$ follows from (2) in the proof of
\prettyref{prop:torus}. We need to show that $\nu(r(a_1,\dots,a_s))$
is again a point-reflection. 
The definition of $\nu(r(a_1,\dots,a_s))$ shows that it sends the
coset 
$$
\left(M_1\cdot\left(\begin{array}{cc}
1&0\\0&\pi_1^{b_1}\end{array}\right),\dots,M_s\cdot
\left(\begin{array}{cc}
1&0\\0&\pi_s^{b_s}\end{array}\right)\right)
$$
to the coset
$$ 
\left(M_1\cdot\left(\begin{array}{cc}
\pi_1^{b_1-2a_1}&0\\0&1\end{array}\right),\dots,M_s\cdot
\left(\begin{array}{cc}
\pi_s^{b_s-2a_s}&0\\0&1\end{array}\right)\right).
$$
This proves the claim.

(3) The equivariance from (1) and (2) shows  that
we obtain a map of quotient cell complexes
$$
\overline{\phi_{\mathcal{L}}}:\mathcal{Z}^s/\mathcal{SN}\to
\op{SL}_2(k[C])\backslash\mathfrak{P}_C.
$$
Surjectivity of $\overline{\phi_{\mathcal{L}}}$ onto the
$\mathcal{L}$-component of $\op{SL}_2(k[C])\backslash\mathfrak{P}_C$
follows as in the proof of \prettyref{prop:torus} from the results in
\prettyref{lem:apartment} and \prettyref{prop:local}. 

(4) For injectivity on $0$-cells, let $(a_1,\dots,a_s)$ and
$(b_1,\dots,b_s)$ be vertices of $\mathcal{Z}^s$ which induce
isomorphic vector bundles
$$
(\mathcal{L}\otimes\mathcal{O}(a_1P_1)\otimes\cdots\otimes
\mathcal{O}(a_sP_s))\oplus\mathcal{L} \textrm{ and }
(\mathcal{L}\otimes\mathcal{O}(b_1P_1)\otimes\cdots\otimes
\mathcal{O}(b_sP_s))\oplus\mathcal{L}.
$$
The isomorphism of vector bundles implies that the line bundles
$\mathcal{O}(a_1P_1)\otimes\cdots\otimes\mathcal{P}(a_sP_s)$ and
$\mathcal{O}(b_1P_1)\otimes\cdots\otimes\mathcal{P}(b_sP_s)$ are
either isomorphic or inverses. The case of isomorphisms is the one
handled in (4) of the proof of \prettyref{prop:torus}. If they are
inverses, then their tensor product is a trivial bundle, and the
corresponding vector bundles are related by a point reflection with
center corresponding to the vector bundle
$(\mathcal{L}\otimes\mathcal{O}((a_1+b_1)P_1)\otimes\cdots\otimes  
\mathcal{O}((a_s+b_s)P_s))\oplus\mathcal{L}$. The tuples
$(a_1,\dots,a_s)$ and $(b_1,\dots,b_s)$ are then
$\mathcal{SN}$-conjugate, and we get injectivity on
$0$-cells. Injectivity for $n$-cells then follows again from
\prettyref{prop:local}.

(5) Finally, the stabilizer statements follow from
\prettyref{prop:aut} and the fact that all vector bundles are split
with non-isomorphic direct summands. 
\end{proof}

\subsection{Functoriality}
From \prettyref{prop:functor1}, we know that morphisms of curves
$f:\overline{C}\to\overline{D}$ induce morphisms of parabolic
subcomplexes and parabolic quotients. Now that we have precise
knowledge of the structure of the parabolic quotients, we can also
describe precisely the effect of the induced morphisms: 

\begin{proposition}
\label{prop:functor2}
Let $k$ be an algebraically closed field. 
Let $f:\overline{D}\rightarrow \overline{C}$ be a finite morphism of
smooth projective curves over $k$, let $P_1,\dots,P_s$ be
points on $\overline{C}$, and let $Q_1,\dots,Q_t$ be points on
$\overline{D}$ not in the preimage of the $P_i$. 
Set $C=\overline{C}\setminus\{P_1,\dots,P_s\}$ and
$D=\overline{D}\setminus(\{f^{-1}(\{P_1,\dots,P_s\}) \cup
\{Q_1,\dots,Q_t\})$. 

Then we have the following assertions for the induced morphism
$f^\ast:\mathfrak{P}_C\to\mathfrak{P}_D$ of
\prettyref{prop:functor1}: 
\begin{enumerate}
\item 
The composition 
$$
\mathcal{K}(C)\to\pi_0(\op{GL}_2(k[C])\backslash\mathfrak{P}_C)
\stackrel{\pi_0(f^\ast)}{\longrightarrow}
\pi_0(\op{GL}_2(k[D])\backslash\mathfrak{P}_D)\to\mathcal{K}(D)
$$
is induced from pullback of line bundles
$f^\ast:\op{Pic}(C)\to\op{Pic}(D)$, where the first and last map are the
bijections of \prettyref{lem:conn}.
\item The morphism from a given component to the
  image component can be described on model complexes as follows: the
  map $\mathcal{Z}^s\to\mathcal{Z}^{\tilde{s}+t}$ maps the $i$-th
  coordinate diagonally to $f^{-1}(P_i)$ many coordinates of
  $\mathcal{Z}^{\tilde{s}+t}$.
\item The composition
  $k[C]^\times/k^\times\cong\mathcal{T}_C\to\mathcal{T}_D\cong
  k[D]^\times/k^\times$ is induced from the natural map
  $f^\ast:k[C]\to k[D]$  on function rings.
\end{enumerate} 
\end{proposition}

\begin{proof}
(1) is clear and follows because $f^\ast$ on the vertices of the
quotients is given by pullback of vector bundles. 

(2) the map on model complexes is the same one as on the standard
apartment of the  building, and hence the same as the one described in
\prettyref{prop:functor1}. 

(3) the map is induced from the inclusion of the units in the function
fields and the conjugation action on certain apartments. Therefore, it
is the natural map.
\end{proof}

\section{Equivariant homology of the parabolic subcomplex}
\label{sec:parhlgy}

In this section, we will use the description of the quotients
$\Gamma\backslash\mathfrak{P}_C$  to compute the $\Gamma$-equivariant
homology of $\mathfrak{P}_C$, for $\Gamma$ one of the linear groups
$\op{(P)GL}_2(k[C])$ and $\op{(P)SL}_2(k[C])$. The main results are
\prettyref{prop:hlgytorus} and \prettyref{prop:hlgynorm} for the two
types of connected components: the components for non-$2$-torsion line 
bundles contribute the homology of $k[C]^\times$ while the components
for $2$-torsion line bundles contribute groups related to a normalizer
of $k[C]^\times$. Using these explicit descriptions, we can also
completely describe the morphisms on parabolic homology induced from
any quasi-finite morphism of affine curves $f:D\to C$.

\subsection{Recollection on equivariant homology}

We shortly recall equivariant homology and the Borel isotropy spectral
sequence for computing it, cf. \cite[Section VII]{brown:book} or
\cite[Appendix A]{knudson:book}. 

A $G$-complex is a CW-complex $X$ together with an action $G\times
X\rightarrow X$ which is cellular. If $X$ is a $G$-complex, then there
is an action of $G$ on the cellular chain complex $\op{C}_\bullet(X)$,
and the corresponding homology groups 
$$
\op{H}^G_\bullet(X,M):=\op{H}_\bullet(G,\op{C}_\bullet(X)\otimes_{\mathbb{Z}[G]}M)
$$
are called \emph{equivariant homology groups of $(G,X)$}.
If $X$ is contractible (or even just acyclic), then the equivariant
homology of $X$ is group homology, i.e. $\op{H}^G_\bullet(X,M)\cong
\op{H}_\bullet(G,M)$. 

For a $G$-complex $X$, the $G$-equivariant homology of $X$ can be
computed by means of the following spectral sequence, which arises
from the stupid filtration of $\op{C}_\bullet(X)\otimes_{\mathbb{Z}[G]}M$:
$$
E_{p,q}^1=\bigoplus_{\sigma\in\Sigma_p}\op{H}_q(G_\sigma,M_\sigma)\Rightarrow
\op{H}^G_{p+q}(X,M). 
$$
In the above, $\Sigma_p$ denotes a set of representatives for the
$G$-action on $\op{C}_p(X)$. The module $M_\sigma$ is the coefficient
$G$-module $M$ twisted by the orientation character
$\chi_\sigma:G_\sigma\rightarrow\mathbb{Z}/2$ of the stabilizer group
$G_\sigma$. In the special case where each cell stabilizer fixes the
cell pointwise, the orientation character is trivial. This can always
be achieved by a suitable subdivision of cells. 
The first differential is induced from the boundary map of the complex
$X$ and inclusions of stabilizers, cf. \cite[VII.8]{brown:book}:
$$
d^1|_{\op{H}_q(G_\sigma,M_\sigma)}:\op{H}_q(G_\sigma,M_\sigma)\mapsto
\bigoplus_{\tau\subseteq \sigma} \op{H}_q(G_\tau,M_\tau)
$$

In the following we consider the groups $\op{(P)GL}_2(k[C])$
resp. $\op{(P)SL}_2(k[C])$ acting on  the cubical complex
$\mathfrak{P}_C$. 

\subsection{Components of torus type} 

We first evaluate the equivariant homology of the components
corresponding to line bundles $\mathcal{L}$ with
$\mathcal{L}\not\cong\mathcal{L}^{-1}$. These components are called
``of torus type'' for two reasons. On the one hand, the quotient has
the homotopy type of a torus. On the other hand, the equivariant
homology is mainly influenced by the intersection of $\op{GL}_2(k[C])$
with a maximal torus in $\op{GL}_2(K)^s$.

\begin{proposition}
\label{prop:hlgytorus}
Let $k$ be an algebraically closed field, let $\overline{C}$ be a
smooth projective curve and denote
$C=\overline{C}\setminus\{P_1,\dots,P_s\}$. Let $\mathcal{L}$ be a
line bundle on $\overline{C}$ such that
$\mathcal{L}|_C\not\cong\mathcal{L}|_C^{-1}$ and denote by
$\mathfrak{P}_C(\mathcal{L})$ the connected component of
$\mathfrak{P}_C$ corresponding to $\mathcal{L}$. 
\begin{enumerate}
\item 
The $\mathcal{T}$-equivariant inclusion
$\mathcal{Z}^s\to\mathfrak{P}_C$ from \prettyref{prop:torus} induces
isomorphisms on homology  
\begin{eqnarray*}
\op{H}_\bullet^{\op{GL}_2(k[C])}(\mathfrak{P}_C(\mathcal{L}),\mathbb{Z})
&\cong & \op{H}_\bullet((k[C]^\times)^2,\mathbb{Z}) \\
\op{H}_\bullet^{\op{PGL}_2(k[C])}(\mathfrak{P}_C(\mathcal{L}),\mathbb{Z})
&\cong& \op{H}_\bullet(k[C]^\times,\mathbb{Z})
\end{eqnarray*}
\item The $\mathcal{ST}$-equivariant inclusion
  $\mathcal{Z}^s\to\mathfrak{P}_C$ from \prettyref{prop:torus} induces
  an isomorphism on homology
$$\op{H}_\bullet^{\op{(P)SL}_2(k[C])}(\mathfrak{P}_C(\mathcal{L}),\mathbb{Z})
\cong\op{H}_\bullet(k[C]^\times,\mathbb{Z}).
$$
\end{enumerate}
\end{proposition}

\begin{proof}
We give the argument for $\op{SL}_2$, the argument for $\op{GL}_2$ is
similar. Recall from \prettyref{prop:torus} that we have an
equivariant map $\mathcal{Z}^s\to\mathfrak{P}_C(\mathcal{L})$. From
the proof of \prettyref{prop:torus}, we see that the homomorphism
$\tau:\mathcal{T}\to\op{SL}_2(k[C])$ extends to a homomorphism
$\tilde{\tau}:k[C]^\times\to\op{SL}_2(k[C])$. Using the stabilizer
computations of \prettyref{prop:torus}, the map
$\mathcal{Z}^s\to\mathfrak{P}_C(\mathcal{L})$ is still equivariant
with respect to $\tilde{\tau}$, where the action of $k[C]^\times$ on
$\mathcal{Z}^s$ is the one factoring through $k[C]^\times\to
k[C]^\times/k^\times\cong\mathcal{T}$. 

The equivariant map above induces a morphism of isotropy spectral
sequences computing the induced morphism on homology
$$
\op{H}_\bullet^{k[C]^\times}(\mathcal{Z}^s,\mathbb{Z})\to
\op{H}_\bullet^{\op{SL}_2(k[C])}(\mathfrak{P}_C(\mathcal{L}),\mathbb{Z}).
$$
By \prettyref{prop:torus}, the map
$\mathcal{Z}^s\to\mathfrak{P}_C(\mathcal{L})$ induces an isomorphism
of quotient cell complexes. Any cell of $\mathcal{Z}^s$ is stabilized
by $k^\times$, and the cells of $\mathfrak{P}_C(\mathcal{L})$ are
stabilized by groups of the form $k^n\rtimes k^\times$ where the size
of the unipotent group depends on the cell. Note that the action of
$k^\times$ on $k^n$ in the above semi-direct product is induced from
an embedding into a Borel subgroup of $\op{GL}_2(K)$.
For any cell $\sigma$ in $\mathcal{Z}^s$, the composition 
$$
k^\times\cong
\op{Stab}(\sigma;k[C]^\times)\to
\op{Stab}(\sigma;\op{SL}_2(k[C]))\cong k^n\rtimes 
k^\times\stackrel{p}{\longrightarrow} k^\times
$$
is the identity. By \cite[Theorem 4.6.2]{knudson:book}, the inclusion
$$\op{Stab}(\sigma;k[C]^\times)\to \op{Stab}(\sigma;\op{SL}_2(k[C]))$$
induces an isomorphism in group homology. The morphism of equivariant
spectral sequences is then an isomorphism on the $E^1$-page, and hence
induces an isomorphism on equivariant homology.

The claim now follows easily, we simply compute
$\op{H}^{k[C]^\times}_\bullet(\mathcal{Z}^s,\mathbb{Z})$:
the model complex $\mathcal{Z}^s$ is contractible,  
and the action of $k[C]^\times$ factors as a free action
of $k[C]^\times/k^\times$ and a trivial action of $k^\times$. In
particular, we have an isomorphism 
$$
\op{H}^{k[C]^\times}_\bullet(\mathcal{Z}^s,\mathbb{Z})\cong
\op{H}^{k^\times}_\bullet((k[C]^\times/k^\times)\backslash\mathcal{Z}^s,\mathbb{Z}).
$$
Note that the homology of the quotient
$X=(k[C]^\times/k^\times)\backslash\mathcal{Z}^s$ is the homology of the
free abelian group $k[C]^\times/k^\times$, in particular it is
torsion-free. The $k^\times$-equivariant homology of $X$ is defined as
the homology of the Borel construction $k^\times\backslash
(Ek^\times\times X)$. Since $k^\times$ acts trivially on  $X$, this
quotient can be identified as $Bk^\times\times X$. In particular,
applying the K\"unneth formula, we see that the
$k^\times$-equivariant homology of $X$ is the homology of the group
$k[C]^\times$. 
\end{proof}

\subsection{Components of normalizer type, case $\op{SL}_2$} 

We next study the homology of the components corresponding to
$2$-torsion line bundles. These components are called ``of normalizer
type'' since the quotient is (away from $2$) a classifying space for
the intersection of $\op{SL}_2(k[C])$ with the normalizer of a maximal
torus in $\op{SL}_2(k(C))$. The argument in this case is more involved
since there are points which have stabilizer $\op{SL}_2(k)$, which is too
large to just come from the normalizer of a maximal torus. 

We recall from \prettyref{sec:global} what we know about the global
structure of those normalizer type components. We fix a line bundle 
$\mathcal{L}$ with $\mathcal{L}\cong\mathcal{L}^{-1}$ and denote by
$\mathfrak{P}_C(\mathcal{L})$ the connected component of the
parabolic subcomplex $\mathfrak{P}_C$ containing the bundle
$\mathcal{L}\oplus\mathcal{L}$. Combining \prettyref{lem:apartment}
and \prettyref{prop:normalizer}, there is a map
$\phi_{\mathcal{L}}:\mathcal{Z}^s\to \mathfrak{A}_{\mathcal{L}}\subset
\mathfrak{P}_C(\mathcal{L})$ and a 
homomorphism $\nu:\mathcal{SN}\to\op{SL}_2(k[C])$ such that
$\phi_{\mathcal{L}}$ is $\nu$-equivariant and induces an isomorphism
of quotient complexes $\mathcal{Z}^s/\mathcal{SN}\cong
\op{SL}_2(k[C])\backslash\mathfrak{P}_C(\mathcal{L})$. From (1) in the
proof of \prettyref{prop:normalizer}, we see that
$\nu:\mathcal{SN}\to\op{SL}_2(k[C])$ can be extended to a group
homomorphism $\tilde{\nu}:\tilde{\mathcal{SN}}\to\op{SL}_2(k[C])$,
where $\tilde{SN}$ is the group of determinant $1$ monomial $(2\times
2)$-matrices with entries in $k[C]$, i.e.,
$$
\tilde{\mathcal{SN}}=\left\{
\left(\begin{array}{cc}
0&u\\-u^{-1}&0
\end{array}\right)\mid 
u\in k[C]^\times
\right\}\cup \left\{
\left(\begin{array}{cc}
u&0\\0&u^{-1}
\end{array}\right)\mid 
u\in k[C]^\times
\right\}.
$$
The group $\tilde{\mathcal{SN}}$ sits in an extension 
$$
1\to k^\times \to \tilde{\mathcal{SN}}\to \mathcal{SN}\to 1,
$$
the homomorphism $\tilde{\nu}$ is the natural embedding followed by
conjugation as in (1) of \prettyref{prop:normalizer}, and $\nu$ is the
composition of $\tilde{\nu}$ with a choice of section
$\mathcal{SN}\to\tilde{\mathcal{SN}}$. Of course, the
subgroup $k^\times$ above acts trivially on $\mathcal{Z}^s$ and hence
does not contribute to the quotient; its relevance for the isotropy
spectral sequence comes from the contribution to the stabilizer
subgroups. 

\begin{lemma}
\label{lem:ssnorm}
Consider the above action of $\tilde{\mathcal{SN}}$ on
$\mathcal{Z}^s$. With $\mathbb{Z}[1/2]$-coefficients, the
isotropy spectral sequence for the $\tilde{\mathcal{SN}}$-complex
$\mathcal{Z}^s$ degenerates at the $E^2$-page and produces an
isomorphism 
$$
\op{H}_\bullet^{\tilde{\mathcal{SN}}}(\mathcal{Z}^s,\mathbb{Z}[1/2])\cong 
\op{H}_\bullet(\tilde{\mathcal{SN}},\mathbb{Z}[1/2]).
$$
\end{lemma}

\begin{proof}
(1) We first describe the stabilizer subgroups to infer some information
on the page
$E_{p,q}^1=\bigoplus_{\sigma\in\Sigma_p}\op{H}_q(\tilde{\mathcal{SN}}_\sigma,
\mathbb{Z}[1/2])$.
As noted above, the subgroup $k^\times$ (embedded as constant diagonal
matrices) acts trivially, hence every stabilizer contains at least a
copy of $k^\times$. The group $\mathcal{SN}$ sits in an extension 
$$
0\to\mathcal{ST}\to\mathcal{SN}\to k[C]^\times/(k[C]^\times)^2\to 1. 
$$
The subgroup $\mathcal{ST}$ acts via translations on $\mathcal{Z}^s$,
hence acts without fixed points. The square classes in the quotient
act as point inversions with suitable centers, as described in
\prettyref{def:models}. Therefore, the vertices $(a_1,\dots,a_s)$ with
$\sum a_i[P_i]=0$ have stabilizers $k^\times\rtimes\mathbb{Z}/2$ where
the $\mathbb{Z}/2$-factor acts via inversion on $k^\times$. All other
vertices have $k^\times$ as stabilizer. Because the action of the
$\mathbb{Z}/2$-factors on $\mathcal{Z}^s$ is via point inversions, the
higher simplices also all have $k^\times$ as stabilizer.

(2) Note that the identity on $\mathcal{Z}^s$ is equivariant for the
inclusion $\tilde{\mathcal{ST}}\hookrightarrow \tilde{\mathcal{SN}}$,
where $\tilde{\mathcal{ST}}$ is the group of determinant diagonal
$(2\times 2)$-matrices with entries in $k[C]$. The inclusion therefore
induces a morphism of isotropy spectral sequences converging to the
morphism 
$$
\op{H}_\bullet^{\tilde{\mathcal{ST}}}(\mathcal{Z}^s,\mathbb{Z}[1/2])\to 
\op{H}_\bullet^{\tilde{\mathcal{SN}}}(\mathcal{Z}^s,\mathbb{Z}[1/2]).
$$
The first group is the one discussed in the torus case above, and the
isotropy spectral sequence for $\tilde{\mathcal{ST}}$-equivariant
homology of $\mathcal{Z}^s$ degenerates at the $E^2$-page. In
particular, the restrictions of the differentials to the homology of
the stabilizers in $\tilde{\mathcal{ST}}$ is trivial.

(3) Now we note that the quotient map
$\mathcal{Z}^s\to\mathcal{Z}^s/\tilde{\mathcal{ST}}$ is equivariant
for the quotient homomorphism
$$
\tilde{\mathcal{SN}}\to\tilde{\mathcal{SN}}/\tilde{\mathcal{ST}}\cong 
k[C]^\times/(k[C]^\times)^2.
$$
We consider the $k[C]^\times/(k[C]^\times)^2$-equivariant homology of
the quotient $\mathcal{Z}^s/\tilde{\mathcal{ST}}$. Obviously, the
stabilizers in this case are only $2$-torsion. Therefore, the only
non-trivial terms of the spectral sequence for
$k[C]^\times/(k[C]^\times)^2$-equivariant homology of
$\mathcal{Z}^s/\tilde{\mathcal{ST}}$ \emph{with 
  $\mathbb{Z}[1/2]$-coefficients} are concentrated in the
$q=0$-column. Hence the spectral sequence degenerates at the
$E^2$-term, and all further differentials are trivial.  

(4) Combining points (2) and (3) above, we find that all differentials
$d_r$, $r\geq 2$, of the isotropy spectral sequence computing
$\tilde{\mathcal{SN}}$-equivariant homology of $\mathcal{Z}^s$ have to
be trivial: the quotient map from (3) induces an isomorphism on the
$q=0$ part, and everything else is detected on the subgroup
$\tilde{\mathcal{ST}}$ from (2). The claim follows.
\end{proof}

\begin{proposition}
\label{prop:hlgynorm}
Let $k$ be an algebraically closed field, let $\overline{C}$ be a
smooth projective curve and denote
$C=\overline{C}\setminus\{P_1,\dots,P_s\}$. 
Let $\mathcal{L}$ be a line bundle on $\overline{C}$ 
which is $2$-torsion in the
Picard group of $C$, and denote by $\mathfrak{P}_C(\mathcal{L})$ the
corresponding connected component of the parabolic subcomplex,
cf. \prettyref{prop:normalizer}. 

Denote $\tilde{\nu}:\tilde{\mathcal{SN}}\to\op{SL}_2(k[C])$ the group
homomorphism discussed above. The $\tilde{\nu}$-equivariant map 
$\phi_{\mathcal{L}}:\mathcal{Z}^s\to \mathfrak{P}_C(\mathcal{L})$
induces a homomorphism
$$
\op{H}_\bullet(\tilde{\mathcal{SN}},\mathbb{Z}[1/2])\cong 
\op{H}_\bullet^{\tilde{\mathcal{SN}}}(\mathcal{Z}^s,\mathbb{Z}[1/2])
\to
\op{H}_\bullet^{\op{SL}_2(k[C])}(\mathfrak{P}_C(\mathcal{L}),\mathbb{Z}[1/2]) 
$$
such that there is a long exact sequence of $\mathbb{Z}[1/2]$-modules
$$
\cdots\to
\mathcal{RP}^1_{i+1}(k)\otimes_{\mathbb{Z}}\mathbb{Z}[1/2,k[C]^\times/(k[C]^\times)^2]\to 
\op{H}_i(\tilde{\mathcal{SN}},\mathbb{Z}[1/2])\to 
$$
$$
\to
\op{H}_i^{\op{SL}_2(k[C])}(\mathfrak{P}_C(\mathcal{L}),\mathbb{Z}[1/2])
\to 
\mathcal{RP}^1_i(k)\otimes_{\mathbb{Z}}\mathbb{Z}[1/2,k[C]^\times/(k[C]^\times)^2]\to \cdots
$$
where the groups $\mathcal{RP}^1_\bullet(k)$ are the refined scissors
configurations groups recalled in \prettyref{sec:rp1}. 
\end{proposition}

\begin{proof}
(1) As before, the $\tilde{\nu}$-equivariance of the map
$\phi_{\mathcal{L}}:\mathcal{Z}^s\to\mathfrak{P}_C(\mathcal{L})$
follows from  \prettyref{prop:normalizer} together with the remarks
before \prettyref{lem:ssnorm}. This implies the existence of an
induced morphism on homology 
$$
\op{H}_\bullet^{\tilde{\mathcal{SN}}}(\mathcal{Z}^s,\mathbb{Z}[1/2])
\to
\op{H}_\bullet^{\op{SL}_2(k[C])}(\mathfrak{P}_C(\mathcal{L}),\mathbb{Z}[1/2]).
$$
The identification of the source with group homology of
$\tilde{\mathcal{SN}}$ follows from \prettyref{lem:ssnorm}. 
The $\tilde{\nu}$-equivariant map $\mathcal{Z}^s\hookrightarrow
\mathfrak{P}_C(\mathcal{L})$ also induces a morphism of the
corresponding isotropy spectral sequences.
To prove the remaining claims, we investigate this morphism of
isotropy spectral sequences, where we use the following indexing for
the isotropy spectral sequence:
$$
E_{p,q}^1=\bigoplus_{\sigma\in\Sigma_p}\op{H}_q(G_\sigma,M_\sigma)\Rightarrow
\op{H}^G_{p+q}(X,M)
$$ 
with differentials going $d^r_{p,q}:E^r_{p,q}\to E^r_{p-r,q+r-1}$. 
We denote $E^i_{p,q}(\tilde{\mathcal{SN}})$ the $(p,q)$-entry of the
$i$-th page of the spectral sequence for
$\tilde{\mathcal{SN}}$-equivariant homology of $\mathcal{Z}^s$, and
similarly for $E^i_{p,q}(\op{SL}_2(k[C]))$. 

(2) From \prettyref{prop:normalizer}, the map
$\mathcal{Z}^s\to\mathfrak{P}_C(\mathcal{L})$ induces an isomorphism
of quotient cell complexes $\mathcal{Z}^s/\tilde{\mathcal{SN}}\cong
\op{SL}_2(k[C])\backslash\mathfrak{P}_C(\mathcal{L})$.
In particular, the induced morphism
$E^1_{p,0}(\tilde{\mathcal{SN}})\to E^1_{p,0}(\op{SL}_2(k[C]))$
is an isomorphism of complexes and hence induces an isomorphism on
the $(q=0)$-column of the $E^2$-pages.

(3)
Next, we look at the stabilizers. For the
$\tilde{\mathcal{SN}}$-action on $\mathcal{Z}^s$, we discussed in the
proof of \prettyref{lem:ssnorm} that there are
$2^{\operatorname{rk}\mathcal{ST}}$ vertices stabilized by a
semi-direct product $k^\times\rtimes \mathbb{Z}/2$, where the  
quotient is generated by the point inversion. 
For the action of $\op{SL}_2(k[C])$ on $\mathfrak{P}_C(\mathcal{L})$,
there are $2^{\operatorname{rk}\mathcal{T}}$ vertices stabilized by 
$\op{SL}_2(k)$. The map $\mathcal{Z}^s\to\mathfrak{P}_C(\mathcal{L})$
sends the former to the latter types of points, and the induced
morphism on stabilizers is the inclusion of
$k^\times\rtimes\mathbb{Z}/2$ as the normalizer of a maximal torus in
$\op{SL}_2(k)$.  

All the other cells in $\mathcal{Z}^s$ have stabilizer $k^\times$ in
$\tilde{\mathcal{SN}}$. By \prettyref{prop:normalizer}, all other
cells in $\mathfrak{P}_C(\mathcal{L})$ have stabilizer some 
semi-direct product $k^n\rtimes k^\times$ with the unipotent group
$k^n$ depending on the cell.  
As in the torus case, the composition 
$$
k^\times\cong\tilde{\mathcal{SN}}_\sigma\stackrel{\iota}{\longrightarrow}  
(\op{SL}_2(k[C]))_\sigma\cong k^n\rtimes
k^\times\stackrel{p}{\longrightarrow} k^\times
$$
is the identity for any such cell $\sigma$ in $\mathcal{Z}^s$. By
\cite[Theorem 4.6.2]{knudson:book}, the inclusion 
$\tilde{\mathcal{SN}}_\sigma\hookrightarrow (\op{SL}_2(k[C]))_\sigma$
induces an isomorphism in group homology. 
In particular,
the map $\mathcal{Z}^s\to\mathfrak{P}_C(\mathcal{L})$
induces an isomorphism on the $(p\geq 1)$-part of the $E^1$-terms: the
homology of the stabilizers changes only for the 
inclusion $k^\times\rtimes\mathbb{Z}/2\hookrightarrow \op{SL}_2(k)$ of
the normalizer of a maximal torus and this happens for special
$0$-cells only.  

(4) From (3) above, we see that the differentials $d^1_{p\geq
  2,\bullet}$ agree and only the differential
$d^1_{1,\bullet}:\bigoplus_{\sigma\in\Sigma_1}\op{H}_\bullet(G_\sigma)\rightarrow
\bigoplus_{\sigma\in\Sigma_0} \op{H}_\bullet(G_\sigma)$ differs for
the two spectral sequences associated to
$\tilde{\mathcal{SN}}\backslash\mathcal{Z}^s$ and
$\op{SL}_2(k[C])\backslash\mathfrak{P}_C(\mathcal{L})$.  Note that
even this differential is the composition of the differential
$d^1_{1,\bullet}(\tilde{\mathcal{SN}})$ with the map
$\op{H}_\bullet(\op{N}(k))\to\op{H}_\bullet(\op{SL}_2(k))$. In
particular,  the induced morphisms
$E^2_{p,q}(\tilde{\mathcal{SN}})\to E^2_{p,q}(\op{SL}_2(k[C]))$ are
isomorphisms for all $p\geq 2$.

Moreover, one copy of the long exact seqeuence of
\prettyref{prop:rp1exact} at every of the special vertices 
combine to the following long exact sequence
$$
\cdots\to
\mathcal{RP}^1_{i+1}(k)\otimes_{\mathbb{Z}}
\mathbb{Z}[1/2,k[C]^\times/(k[C]^\times)^2]\to 
E^1_{0,i}(\tilde{\mathcal{SN}})\to 
\phantom{some more space}
$$
$$
\phantom{some space}\to E^1_{0,i}(\op{SL}_2(k[C]))
\to 
\mathcal{RP}^1_i(k)\otimes_{\mathbb{Z}}\mathbb{Z}[1/2,k[C]^\times/(k[C]^\times)^2]\to
\cdots 
$$

(5) In the next step, we want to establish an exact sequence
describing the morphism of $E^2$-terms. First note that the map 
$$
E^1_{0,i}(\op{SL}_2(k[C]))\to\mathcal{RP}^1_i(k)\otimes_{\mathbb{Z}}
\mathbb{Z}[1/2,k[C]^\times/(k[C]^\times)^2] 
$$
induces a natural map 
$$
  E^2_{0,i}(\op{SL}_2(k[C]))
\to\mathcal{RP}^1_i(k)\otimes_{\mathbb{Z}} 
\mathbb{Z}[1/2,k[C]^\times/(k[C]^\times)^2] 
$$
because the differential $d^1:E^1_{1,i}\to E^1_{0,i}(\op{SL}_2(k[C]))$
factors through $E^1_{0,i}(\tilde{\mathcal{SN}})$. Moreover, because
the group $E^2_{0,i}$ is a quotient of $E^1_{0,i}$, these maps have
the same image in
$\mathcal{RP}^1_i(k)\otimes_{\mathbb{Z}}
\mathbb{Z}[1/2,k[C]^\times/(k[C]^\times)^2]$.
The elements in $E^1_{0,i}(\tilde{\mathcal{SN}})$ mapping to $0$
in $E^2_{0,i}(\tilde{\mathcal{SN}})$ come from
$E^1_{1,i}(\tilde{\mathcal{SN}})\cong E^1_{1,i}(\op{SL}_2(k[C]))$. The
elements in $\ker(E^1_{0,i}(\tilde{\mathcal{SN}})\to
E^2_{0,i}(\tilde{\mathcal{SN}}))$
can then be identified exactly with the elements of 
$\op{coker}\left(E^2_{1,i}(\tilde{\mathcal{SN}})\to
  E^2_{1,i}(\op{SL}_2(k[C]))\right)$, whence we get an exact sequence 
$$
0\to \op{coker}\left(E^2_{1,i}(\tilde{\mathcal{SN}})\to
  E^2_{1,i}(\op{SL}_2(k[C]))\right) \to 
E^1_{0,i}(\tilde{\mathcal{SN}})\to E^2_{0,i}(\tilde{\mathcal{SN}})\to
0. 
$$
There is a similar exact sequence 
\begin{eqnarray*}
0&\to& \op{coker}\left(E^2_{1,i}(\tilde{\mathcal{SN}})\to
  E^2_{1,i}(\op{SL}_2(k[C]))\right) \\&\to& 
\ker\left(E^1_{0,i}(\tilde{\mathcal{SN}})\to
  E^1_{0,i}(\op{SL}_2(k[C]))\right) \\&\to&
\ker\left(E^2_{0,i}(\tilde{\mathcal{SN}})\to E^2_{0,i}(\op{SL}_2(k[C]))\right)
\to
0. 
\end{eqnarray*}
The exact sequence of (4) then gives rise to an exact sequence 
\begin{eqnarray*}
\cdots&\to&
\mathcal{RP}^1_{i+1}(k)\otimes_{\mathbb{Z}}
\mathbb{Z}[1/2,k[C]^\times/(k[C]^\times)^2]\\&\to& 
\ker\left(E^1_{0,i}(\tilde{\mathcal{SN}})\to
  E^1_{0,i}(\op{SL}_2(k[C]))\right) \\&\to&
\op{coker}\left(E^2_{0,i}(\tilde{\mathcal{SN}})\to
  E^2_{0,i}(\op{SL}_2(k[C]))\right) 
\\&\to& 
\mathcal{RP}^1_i(k)\otimes_{\mathbb{Z}}\mathbb{Z}[1/2,k[C]^\times/(k[C]^\times)^2]\to
\cdots 
\end{eqnarray*}
and the second term is also completely decomposed into stuff coming
from the comparison of $E^2$-terms, using the exact sequence just above.

Finally, we investigate the map $E^2_{1,i}(\tilde{\mathcal{SN}})\to
E^2_{1,i}(\op{SL}_2(k[C]))$. Let $\sigma\in
E^2_{1,i}(\tilde{\mathcal{SN}})$ be such that it becomes trivial in
$E^2_{1,i}(\op{SL}_2(k[C]))$. There is a representative
$\tilde{\sigma}\in E^1_{1,i}(\tilde{\mathcal{SN}})$ with
$d^1\tilde{\sigma}=0$. Since the maps 
$E^1_{p,i}(\tilde{\mathcal{SN}})\to E^1_{p,i}(\op{SL}_2(k[C]))$ are 
isomorphisms for any $p\geq 1$ and any $i$, if the image of $\sigma$
in $E^2_{1,i}(\op{SL}_2(k[C]))$ is trivial, then already
$\sigma=0$. We see that the map $E^2_{1,i}(\tilde{\mathcal{SN}})\to
E^2_{1,i}(\op{SL}_2(k[C]))$ is injective. 

Recall again that all the comparison morphisms
$E^1_{p,i}(\tilde{\mathcal{SN}})\to E^1_{p,i}(\op{SL}_2(k[C]))$ are
isomorphisms for all $i$ and $p\geq 1$. In particular, the comparison
maps on the $E^2$ terms discussed above (dealing with kernels and
cokernels of the maps induced on $E^2_{0,i}$ and $E^2_{1,i}$) are the
only ones which have a possibly non-trivial kernel or cokernel. In
particular, the long exact sequence above gives a complete comparison
statement for the $E^2$-terms of the two spectral sequences. 

(6) By \prettyref{lem:ssnorm}, the spectral sequence for 
$SL_2(k[C])$-equivariant homology of $\mathfrak{P}(\mathcal{L})$ also
has to degenerate at the $E^2$-term, because all the differentials
(except the $d^1_{1,i}:E^1_{1,i}\to E^1_{0,i}$ discussed above) come
from the spectral sequence for $\tilde{\mathcal{SN}}$. In particular,
$E^\infty_{p,q}(\op{SL}_2(k[C]))=E^2_{p,q}(\op{SL}_2(k[C]))$ and
$E^\infty_{p,q}(\op{SL}_2(k[C]))\cong E^\infty_{p,q}(\tilde{\mathcal{SN}})$
for $p\geq 2$.

(7) We now want to promote the exact sequence on $E^2$-terms to an
exact sequence of the homology groups proper. Obviously, the
equivariant map of the model complex into the parabolic component
induces a morphism of homology groups
$\op{H}_\bullet(\tilde{\mathcal{SN}},\mathbb{Z}[1/2])\to
\op{H}_\bullet^{\op{SL}_2(k[C])}(\mathfrak{P}_C(\mathcal{L}),\mathbb{Z}[1/2])$
which is compatible with the filtrations and the spectral
sequences. The degeneration of the spectral sequence in (6) together
with the computations of kernels and cokernels of comparison maps in
step (5) imply that 
$$
\op{ker}\left(\op{H}_i(\tilde{\mathcal{SN}},\mathbb{Z}[1/2])\to
\op{H}_i^{\op{SL}_2(k[C])}(\mathfrak{P}_C(\mathcal{L}),\mathbb{Z}[1/2])\right)
$$
$$\cong \ker\left(E^2_{0,i}(\tilde{\mathcal{SN}})\to
  E^2_{0,i}(\op{SL}_2(k[C]))\right). 
$$
Using this identification, there is a natural map 
$$
\mathcal{RP}^1_{i+1}(k)[k[C]^\times/(k[C]^\times)^2]\to 
\op{ker}\left(\op{H}_i(\tilde{\mathcal{SN}},\mathbb{Z}[1/2])\to
\op{H}_i^{\op{SL}_2(k[C])}(\mathfrak{P}_C(\mathcal{L}),\mathbb{Z}[1/2])\right)
$$
coming from the composition
$$
\mathcal{RP}^1_{i+1}(k)[k[C]^\times/(k[C]^\times)^2]\to
E^1_{0,i}(\tilde{\mathcal{SN}})\to E^2_{0,i}(\tilde{\mathcal{SN}}),
$$
cf. step (5) above. 

From the analysis of the complex $\mathfrak{P}_C(\mathcal{L})$ we see
that its equivariant homology is detected by the maximal torus and the
subgroups $\op{SL}_2(k)$ sitting at the special vertices. In
particular, there is a natural map 
$$
\op{H}_\bullet(\op{SL}_2(k),\mathbb{Z}[1/2])[k[C]^\times/(k[C]^\times)^2]\to
\mathcal{RP}^1_\bullet(k)[k[C]^\times/(k[C]^\times)^2]
$$
and the analysis of the spectral sequence shows that this map extends
to the required natural map 
$$
\op{H}_i^{\op{SL}_2(k[C])}(\mathfrak{P}_C(\mathcal{L}),\mathbb{Z}[1/2])\to 
\mathcal{RP}^1_{i}(k)[k[C]^\times/(k[C]^\times)^2]
$$
because all the additional relations come from homology of $\op{N}(k)$
and then are satisfied in the target as well. 

Using step (5) again, we can describe the cokernel of the comparison
map via an extension
$$0\to \op{coker}\left(
E^2_{0,i}(\tilde{\mathcal{SN}})\to E^2_{0,i}(\op{SL}_2(k[C]))
\right)\to
$$
$$
\to\op{coker}\left(\op{H}_i(\tilde{\mathcal{SN}},\mathbb{Z}[1/2])\to
\op{H}_i^{\op{SL}_2(k[C])}(\mathfrak{P}_C(\mathcal{L}),\mathbb{Z}[1/2])\right)
\to
$$
$$
\to\op{coker}\left(
E^2_{1,i-1}(\tilde{\mathcal{SN}})\to E^2_{1,i-1}(\op{SL}_2(k[C]))\right)\to 0.
$$
Also from step (5), we know that there is a natural map
$\mathcal{RP}^1_{i}(k)[k[C]^\times/(k[C]^\times)^2]\to
E^2_{0,i}(\tilde{\mathcal{SN}})$
whose kernel also sits in an exact sequence with the same outer terms
as above. The natural map induces isomorphisms on the graded pieces,
by the computations in step (5). This implies the exactness claim. 
\end{proof}

\subsection{Components of normalizer type, case $\op{PGL}_2$} 

Now we state the corresponding results on the homology of the
normalizer type components for the case of the group
$\op{(P)GL}_2$. 

\begin{proposition}
\label{prop:hlgynormgl2}
Let $k$ be an algebraically closed field, let $\overline{C}$ be a
smooth projective curve and denote
$C=\overline{C}\setminus\{P_1,\dots,P_s\}$. 
Let $\mathcal{L}$ be a line bundle on $\overline{C}$ which is
$2$-torsion in the Picard group of $C$, and denote by
$\mathfrak{P}_C(\mathcal{L})$ the 
corresponding connected component of the parabolic subcomplex,
cf. \prettyref{prop:normalizer}. 

Denote $\tilde{\nu}:\tilde{\mathcal{N}}\to\op{PGL}_2(k[C])$ the
analogue of the group homomorphism discussed above. The
$\tilde{\nu}$-equivariant map  
$\phi_{\mathcal{L}}:\mathcal{Z}^s\to \mathfrak{P}_C(\mathcal{L})$
induces a homomorphism
$$
\op{H}_\bullet(\tilde{\mathcal{N}},\mathbb{Z}[1/2])\cong 
\op{H}_\bullet^{\tilde{\mathcal{N}}}(\mathcal{Z}^s,\mathbb{Z}[1/2])
\to
\op{H}_\bullet^{\op{PGL}_2(k[C])}(\mathfrak{P}_C(\mathcal{L}),\mathbb{Z}[1/2]) 
$$
such that there is a long exact sequence of $\mathbb{Z}[1/2]$-modules
$$
\cdots\to
\mathcal{P}^1_{i+1}(k)\to
\op{H}_i(\tilde{\mathcal{N}},\mathbb{Z}[1/2])\to 
\op{H}_i^{\op{PGL}_2(k[C])}(\mathfrak{P}_C(\mathcal{L}),\mathbb{Z}[1/2])
\to 
\mathcal{P}^1_i(k)\to \cdots
$$
\end{proposition}

The arguments are similar to the ones used for $\op{SL}_2$ above, and
we omit a detailed proof: the proof first proceeds through an analogue
of  \prettyref{lem:ssnorm} for the group $\tilde{\mathcal{N}}$, and
then mainly uses the $\op{PGL}_2$-part of \prettyref{prop:normalizer}
in much the  same way as happened in the proof of
\prettyref{prop:hlgynorm}.  

The groups $\mathcal{P}^1_i(k)$ appearing above are the generalized
scissors congruence groups of \cite{dupont}, cf. also
\prettyref{sec:rp1}. Using the long exact sequence connecting homology
of the normalizer $\op{N}(k)$, homology of $\op{SL}_2(k)$ and the
scissors congruence groups $\mathcal{P}^1_i(k)$, we can simplify the
above result further and describe the parabolic homology as the
following pushout, all with $\mathbb{Z}[1/2]$-coefficients:
$$
\op{H}_\bullet^{\op{PGL}_2(k[C])}(\mathfrak{P}_C(\mathcal{L}))\cong 
\op{H}_\bullet(\tilde{\mathcal{N}})\oplus_{\op{H}_\bullet(\op{N}(k))}
\op{H}_\bullet(\op{PGL}_2(k)). 
$$

\subsection{Functoriality} 
We have described the $\op{SL}_2(k[C])$-equivariant homology of 
the parabolic subcomplex $\mathfrak{P}_C$ above. We now want to
describe the behaviour of parabolic homology under morphisms of
curves. In particular we want to compute the colimit of 
these homologies to $\op{SL}_2(k(C))$ and $\op{SL}_2(\overline{k(C)})$. 

\begin{proposition}
\label{prop:hlgyfun}
Let $k$ be an algebraically closed field. 
Let $f:\overline{D}\rightarrow \overline{C}$ be a finite morphism of
smooth projective curves over $k$, let $P_1,\dots,P_s$ be
points on $\overline{C}$, and let $Q_1,\dots,Q_t$ be points on
$\overline{D}$ not in the preimage of the $P_i$. 
Set $C=\overline{C}\setminus\{P_1,\dots,P_s\}$ and
$D=\overline{D}\setminus(\{f^{-1}(\{P_1,\dots,P_s\}) \cup
\{Q_1,\dots,Q_t\})$. Denote by $\mathfrak{P}_C$ and $\mathfrak{P}_D$
the parabolic subcomplexes in $\mathfrak{X}_C$ and $\mathfrak{X}_D$,
respectively. 
Then we have the following assertions for the morphism
$$f^\ast:\op{H}_\bullet^{\op{SL}_2(k[C])}(\mathfrak{P}_C)\rightarrow
\op{H}_\bullet^{\op{SL}_2(k[D])}(\mathfrak{P}_D)$$ induced from the
morphism of \prettyref{prop:functor1}, where we use
$\mathbb{Z}[1/2]$-coefficients throughout: 

\begin{enumerate}
\item 
The composition 
$$
\mathcal{K}(C)\to\pi_0(\op{SL}_2(k[C])\backslash\mathfrak{P}_C)
\stackrel{\pi_0(f^\ast)}{\longrightarrow}
\pi_0(\op{SL}_2(k[D])\backslash\mathfrak{P}_D)\to\mathcal{K}(D)
$$
is induced from pullback of line bundles
$f^\ast:\op{Pic}(C)\to\op{Pic}(D)$, where the first and last map are the
bijections of \prettyref{lem:conn}.

\item For each line bundle $\mathcal{L}$ on $\overline{C}$, there is
  an induced
  map 
  $$f^\ast_{\mathcal{L}}:\op{H}_\bullet^{\op{SL}_2(k[C])}(\mathfrak{P}_C(\mathcal{L})) 
  \rightarrow 
  \op{H}_\bullet^{\op{SL}_2(k[D])}(\mathfrak{P}_D(f^\ast(\mathcal{L}))).$$ 
The map $f^\ast$ is the direct sum of these. 
\item If $\mathcal{L}$ is not $2$-torsion in $\op{Pic}(D)$, then
  $f^\ast_{\mathcal{L}}$ is identified via the isomorphisms of
  \prettyref{prop:hlgytorus} with the natural map
$$
\op{H}_\bullet(k[C]^\times)\to\op{H}_\bullet(k[D]^\times).
$$
\item If $\mathcal{L}$ is not $2$-torsion in $\op{Pic}(C)$ but becomes
  $2$-torsion in $\op{Pic}(D)$, then $f^\ast_{\mathcal{L}}$ is
  identified with the composition
$$
\op{H}_\bullet(k[C]^\times)\to \op{H}_\bullet(\tilde{\mathcal{SN}})\to
\op{H}_\bullet^{\op{SL}_2(k[C])}(\mathfrak{P}_D(f^\ast\mathcal{L})),
$$
where the first map is the inclusion of diagonal matrices into
monomial matrices, and the second map is the morphism from
\prettyref{prop:hlgynorm}. 
\item If $\mathcal{L}$ is $2$-torsion in $\op{Pic}(C)$, then  
  $f^\ast_{\mathcal{L}}$ sits in a commutative ladder connecting the
  respective long exact sequences of \prettyref{prop:hlgynorm} for
  $\mathfrak{P}_C$ and $\mathfrak{P}_D$. The other two types of maps
  of the commutative ladder, 
$\op{H}_\bullet(\tilde{\mathcal{SN}}_C)\to
\op{H}_\bullet(\tilde{\mathcal{SN}}_D)$ and 
$$
\mathcal{RP}^1_\bullet(k)\otimes_{\mathbb{Z}}
\mathbb{Z}[1/2,k[C]^\times/(k[C]^\times)^2]\to 
\mathcal{RP}^1_\bullet(k)\otimes_{\mathbb{Z}}
\mathbb{Z}[1/2,k[D]^\times/(k[D]^\times)^2]
$$
are induced by the natural map $k[C]^\times\to k[D]^\times$. 
\end{enumerate} 
\end{proposition}

\begin{proof}
The result is a direct consequence of \prettyref{prop:functor2} and the
computations in \prettyref{prop:hlgytorus} and
\prettyref{prop:hlgynorm}. 
\end{proof}

Obviously, there is a similar functoriality result for the case
$\op{PGL}_2$ where again all relevant maps are induced from the
natural one $f^\ast:k[C]^\times\to k[D]^\times$. 

\begin{definition}
Let $k$ be an algebraically closed field, let $\overline{C}$ be a
smooth projective curve over $k$, and set
$C=\overline{C}\setminus\{P_1,\dots,P_s\}$. We define the
\emph{parabolic homology} of $\op{SL}_2(k(C))$ to be 
$$
\widehat{\op{H}}_\bullet(\op{SL}_2(k(C)),\mathbb{Z}/\ell)=
\colim_{S\subseteq\overline{C}(k)}
\op{H}_\bullet^{\op{SL}_2(k[\overline{C}\setminus
  S])}(\mathfrak{P}_{\overline{C}\setminus S},\mathbb{Z}/\ell),
$$
where the colimit is taken over all finite sets $S$ of closed points
of $\overline{C}$, ordered by inclusion.
\end{definition}

As a direct consequence of \prettyref{prop:hlgyfun}, we can describe
the parabolic homology for function fields of curves. 

\begin{proposition}
\label{prop:fthlgy}
Let $k$ be an algebraically closed field, let $C$ be a smooth curve
over $k$ and let $\ell$ be an odd prime different from the characteristic
of $k$. Then we have the following exact sequence 
$$
\cdots\to
\mathcal{RP}^1_{i+1}(k,\mathbb{Z}/\ell)[k(C)^\times/(k(C)^\times)^2]\to 
\op{H}_i(\op{N}(k(C)),\mathbb{Z}/\ell)\to 
$$
$$
\to
\widehat{\op{H}}_i(\op{SL}_2(k(C)),\mathbb{Z}/\ell)
\to 
\mathcal{RP}^1_i(k,\mathbb{Z}/\ell)[k(C)^\times/(k(C)^\times)^2]\to
\cdots, 
$$
where $\op{N}(k(C))$ denotes the normalizer of a maximal torus in
$\op{SL}_2(k(C))$. 
\end{proposition}

\subsection{Low-dimensional example calculations}

We discuss low-dimensional special cases of the above computations in
\prettyref{prop:hlgytorus} and \prettyref{prop:hlgynorm}. In low
degrees, we can say more about the precise relation between the
homology of $\op{SL}_2(k)$ and the homology of the normalizer of the
maximal torus $\op{N}(k)$.

We start with the degree $1$ case: 

\begin{corollary}
Let $k$ be an algebraically closed field, let $\overline{C}$ be a
smooth projective curve and denote
$C=\overline{C}\setminus\{P_1,\dots,P_s\}$. 
There is an isomorphism
$$
\op{H}_1^{\op{SL}_2(k[C])}(\mathfrak{P}_C,\mathbb{Z}[1/2])\cong 
\bigoplus_{[\mathcal{L}]\in\mathcal{K}(C), 2\mathcal{L}\neq 0}
k[C]^\times\otimes_{\mathbb{Z}}\mathbb{Z}[1/2].
$$
In the limit, we have
$\widehat{\op{H}}_1(\op{SL}_2(k(C)),\mathbb{Z}[1/2])=0$.
\end{corollary}

\begin{proof}
By \prettyref{prop:hlgytorus}, if $\mathcal{L}$ is a line bundle on
$\overline{C}$ which is not $2$-torsion in $\op{Pic}(C)$, then
$$
\op{H}_1^{\op{SL}_2(k[C])}(\mathfrak{P}_C(\mathcal{L}),\mathbb{Z}[1/2])\cong
k[C]^\times\otimes_{\mathbb{Z}}\mathbb{Z}[1/2]. 
$$

Now let $\mathcal{L}$ be a line bundle on $\overline{C}$ which is
$2$-torsion in $\op{Pic}(C)$. We investigate the components of the
exact sequence computing
$\op{H}_1^{\op{SL}_2(k[C])}(\mathfrak{P}_C(\mathcal{L}),\mathbb{Z}[1/2])$.  
The group $\tilde{\mathcal{SN}}$ is an extension of the free abelian
group $\tilde{\mathcal{ST}}$ by a $2$-torsion group $G$. Obviously,
$\op{H}_1(G,\op{H}_0(\tilde{\mathcal{ST}},\mathbb{Z}[1/2]))=0$. But 
$$
\op{H}_0(G,\op{H}_1(\tilde{\mathcal{ST}},\mathbb{Z}[1/2]))\cong 
\op{H}_0(G,\tilde{\mathcal{ST}}\otimes_{\mathbb{Z}}\mathbb{Z}[1/2])=0
$$
because any element of $G$ acts by inversion followed by
multiplication with some element, whence the coinvariants are
$2$-torsion. The Hochschild-Serre spectral sequence for the extension
then implies
$\op{H}_1(\tilde{\mathcal{SN}},\mathbb{Z}[1/2])=0$. Similarly, the 
computations in \cite{hutchinson:bw} imply that
$\mathcal{RP}^1_1(k)=0$, cf. \prettyref{sec:rp1}, and therefore
$$\mathcal{RP}^1_1(k)\otimes_{\mathbb{Z}}
\mathbb{Z}[1/2,k(C)^\times/(k(C)^\times)^2]=0.
$$
The exact sequence of \prettyref{prop:hlgynorm} then implies
$\op{H}_1^{\op{SL}_2(k[C])}(\mathfrak{P}_C(\mathcal{L}),\mathbb{Z}[1/2])=0$. 

The claim then follows from \prettyref{lem:conn}. The limit claim
follows from the above and \prettyref{prop:fthlgy}.
\end{proof}

\begin{corollary}
Let $k$ be an algebraically closed field, let $\overline{C}$ be a
smooth projective curve and denote
$C=\overline{C}\setminus\{P_1,\dots,P_s\}$. 
There is an exact sequence 
$$
0\to\op{H}_3(\tilde{\mathcal{SN}},\mathbb{Z}[1/2]) \to
\op{H}_3^{\op{SL}_2(k[C])}(\mathfrak{P}_C(\mathcal{L}),\mathbb{Z}[1/2])\to $$
$$
\to
\mathcal{RP}^1_3(k)\otimes_{\mathbb{Z}}\mathbb{Z}[1/2,k[C]^\times/(k[C]^\times)^2]\to 
$$
$$
\to \op{H}_2(\tilde{\mathcal{SN}},\mathbb{Z}[1/2])\to
\op{H}_2^{\op{SL}_2(k[C])}(\mathfrak{P}_C(\mathcal{L}),\mathbb{Z}[1/2])\to 0.
$$
The maps are the ones appearing in the Bloch-Wigner exact sequence,
cf. \cite{hutchinson:bw}.
\end{corollary}

\appendix
\section{Refined scissors congruence groups}
\label{sec:rp1}

The following section provides a recollection on the exact
sequence relating homology of $\op{SL}_2(k)$, homology of the
normalizer of a maximal torus $\op{N}(k)$ and suitable generalizations
of pre-Bloch groups $\mathcal{RP}^1_\bullet(k)$. The material is 
well-known from the work of Bloch, Suslin, Dupont, Sah, Hutchinson and
others, cf. \cite{suslin:k3}, \cite{dupont},
\cite{hutchinson:bw}. As the statements are somehow scattered over the
literature and not usually stated in the generality needed, we provide
a detailed review of the definitions of the refined scissors
congruence groups (or point configuration groups)
$\mathcal{RP}^1_\bullet(k)$ and the proof of the relevant exact
sequences for homology of $\op{SL}_2(k)$. 

\subsection{Points on the projective line}

Let $F$ be an infinite field. Consider the standard action of
the general linear group $\op{GL}_2(F)$ on the $F$-points of the
projective line $\mathbb{P}^1(F)$ given by 
$$
\left(\begin{array}{cc}
a&b\\c&d
\end{array}\right)\cdot z\mapsto \frac{az+b}{cz+d}.
$$
As the center obviously acts trivially, the group action factors
through $\op{PGL}_2(F)$. Restriction to determinant $1$ provides the
standard group actions of $\op{SL}_2(F)$ and $\op{PSL}_2(F)$ on
$\mathbb{P}^1(F)$.

\begin{definition}
Let $F$ be an infinite field. 
\begin{enumerate}
\item Denote by $C_\bullet(F)$ the \emph{complex of points on
    $\mathbb{P}^1$}, which in degree $n$ has the free abelian group
  $C_n(F)$ generated by $(n+1)$-tuples $(x_0,\dots,x_n)$ of distinct
  points on $\mathbb{P}^1(F)$.  
\item Denote by $C_\bullet^{\op{alt}}(F)$ the \emph{alternating complex of
  points on $\mathbb{P}^1$}, where $C_n^{\op{alt}}(F)$ is the free
  abelian group generated by $(n+1)$-tuples $(x_0,\dots,x_n)$ of
  points on $\mathbb{P}^1(F)$ modulo the identifications 
$$
(x_{\pi(0)},\dots,x_{\pi(n)})=\op{sgn}\pi\cdot (x_0,\dots,x_n),
$$
where $\pi$ is any permutation of the set of indices $\{0,\dots,n\}$.
\end{enumerate}
In both the above cases, the differential is given by the obvious 
$$
d(x_0,\dots,x_n)=\sum_{i=0}^n(-1)^i(x_0,\dots,\widehat{x_i},\dots,x_n).
$$
The complexes are augmented by mapping 
$$
C_0(F)=C_0^{\op{alt}}(F)\to\mathbb{Z}:(x_0)\mapsto 1.
$$
\end{definition}

We recall the well-known acyclicity lemma, cf. \cite[Lemma
2.3.2]{knudson:book}, \cite[Lemma 3.6]{dupont} or any of a huge number
of further possible sources.

\begin{lemma}
The augmented complexes $C^{(\op{alt})}_\bullet(F)\to \mathbb{Z}\to 0$
of points on $\mathbb{P}^1$ are acyclic. 
\end{lemma}

\begin{proof}
Let $\sum_{j=1}^Nn_j(x_{j,0},\dots,x_{j,q})$ be a $q$-cycle in
$C_\bullet(F)$. Since $F$ is infinite, there exists a point
$x\in\mathbb{P}^1(F)$ different from any of the points $x_{j,q}$. Then 
$$
d\left(\sum_{j=1}^Nn_j(x,x_{j,0},\dots,x_{j,q})\right)=
\sum_{j=1}^Nn_j(x_{j,0},\dots,x_{j,q}).
$$

For the alternating complexes, one can even choose a global base-point
$x\in\mathbb{P}^1$ and use the global contraction 
$$
s_q:C^{\op{alt}}_q(F)\to C^{\op{alt}}_{q+1}(F):(x_0,\dots,x_q)\mapsto
(x,x_0,\dots,x_q), s_{-1}(1)=(x).
$$
\end{proof}

Contractibility implies that the $\op{SL}_2(F)$-equivariant homology
of the complexes of points on $\mathbb{P}^1$ can indeed be identified
with group homology of $\op{SL}_2(F)$. 

\begin{corollary}
\label{cor:acyclic}
Let $F$ be an infinite field, and let $\Gamma$ be any one of the
groups $\op{(P)GL}_2(F)$ or $\op{(P)SL}_2(F)$. Then the augmentation
induces isomorphisms 
$$
\op{H}_\bullet(\Gamma,C^{(\op{alt})}_\bullet(F))
\stackrel{\cong}{\longrightarrow} \op{H}_\bullet(\Gamma,\mathbb{Z}).
$$
\end{corollary}

The hyperhomology spectral sequence associated to
$\op{H}_\bullet(\Gamma,C_\bullet^{(\op{alt})})$ provides the relation between
homology of $\Gamma$, the normalizer of a maximal torus in $\Gamma$
and the refined scissors congruence groups. 

There is a  notion of decomposability of point configurations in
projective $n$-space, cf. \cite[p. 126]{dupont}. Using
$\mathbb{Z}[1/2]$-coefficients and noticing that $(x,x)=-(x,x)$
implies that $(x,x)$ is $2$-torsion, we can simplify the notion of
decomposability for our purposes: for the projective
line, an $(n+1)$-tuple of points $(x_0,\dots,x_n)$ is decomposable if
and only if its support contains at most two points. In particular, a
non-trivial element in $C^{(\op{alt})}_q(F)$ is decomposable if and
only if $q\leq 1$. The \emph{decomposable subcomplex of the complex of
  points} is defined as follows:
$$
F_0C^{(\op{alt})}_q(F)=\left\{\begin{array}{ll}
C^{(\op{alt})}_q(F) & q\leq 1\\
0  &\textrm{otherwise}
\end{array}\right.
$$
It is obviously stable under the action of $\op{(P)GL}_2(F)$.

\begin{definition}
\label{def:rp1}
Let $F$ be an infinite field. The \emph{refined scissors congruence
  groups} (or \emph{refined point configuration groups}) are defined as
$$
\mathcal{RP}^1_\bullet(F):=
\op{H}_\bullet(\op{SL}_2(F),C^{\op{alt}}_{\bullet}(F)/F_0C^{\op{alt}}_\bullet(F))
$$
The \emph{scissors congruence groups} (or \emph{point configuration
  groups}) are defined as 
$$
\mathcal{P}^1_\bullet(F):=
\op{H}_\bullet(\op{PGL}_2(F),C^{\op{alt}}_{\bullet}(F)/F_0C^{\op{alt}}_\bullet(F))
$$
\end{definition}

\begin{remark}
For considerations with finite coefficients, we can define similar
groups 
$$
\mathcal{RP}^1_\bullet(F,\mathbb{Z}/\ell):=
\op{H}_\bullet(\op{SL}_2(F),
C^{\op{alt}}_{\bullet}(F)/F_0C^{\op{alt}}_\bullet(F)
\otimes_{\mathbb{Z}}\mathbb{Z}/\ell) 
$$
\end{remark}

\begin{remark}
The notation $\mathcal{P}^1_q(F)$ is taken from \cite{dupont}. The
notation $\mathcal{RP}^1_q(F)$ for the refined scissors congruence
groups is a combination of the notation from \cite{dupont} and 
\cite{hutchinson:bw}. In \cite{hutchinson:bw}, only the groups
$\mathcal{RP}^1_3(F)$ are considered (and simply denoted
$\mathcal{RP}_1(F)$). 
\end{remark}

\begin{proposition}
If $F$ is a quadratically closed field,  the natural map
$\op{SL}_2(F)\to\op{PGL}_2(F)$ induces isomorphisms 
$$
\mathcal{RP}^1_\bullet(F)\otimes_{\mathbb{Z}}\mathbb{Z}[1/2]
\stackrel{\cong}{\longrightarrow} 
\mathcal{P}^1_\bullet(F)\otimes_{\mathbb{Z}}\mathbb{Z}[1/2].
$$
\end{proposition}

\begin{proof}
The natural map factors as
$\op{SL}_2(F)\to\op{PSL}_2(F)\to\op{PGL}_2(F)$, and we will show that
any of these homomorphisms induces an isomorphism on equivariant
homology with coefficients in $C^{\op{alt}}_\bullet(F)$. 

If $F$ is quadratically closed, then $\op{PSL}_2(F)$ and
$\op{PGL}_2(F)$ are in fact the same groups, because any matrix
$M\in\op{PGL}_2(F)$ can be scaled, using the diagonal matrix
$\op{diag}(\det M^{\frac{1}{2}},\det M^{\frac{1}{2}})$, to a matrix in
$\op{PSL}_2(F)$. Therefore, the second homomorphism obviously induces
an isomorphism on homology. 

The homomorphism $\op{SL}_2(F)\to\op{PSL}_2(F)$ is
surjective and has kernel $\{\pm
I\}\cong\mathbb{Z}/2\mathbb{Z}$. Since the kernel acts in fact
trivially on the complexes $C^{\op{alt}}_\bullet(F)$, the projection 
hence induces an isomorphism on homology with
$\mathbb{Z}[1/2]$-coefficients. 
\end{proof}

\subsection{The hyperhomology spectral sequence}

The hyperhomology spectral sequence for the action of $\op{SL}_2(F)$
on the complexes $C^{(\op{alt})}_\bullet(F)$ provides some insights
into the structure of the homology of $\op{SL}_2(F)$, which we discuss
next. 

We first recall a suitable resolution of the complex
$F_0C^{\op{alt}}_\bullet(F)$ from \cite[Proposition 13.22]{dupont}. 
The complex $D_\bullet(F)$ is the following complex
$$
D_\bullet(F)=\bigoplus_{x\in\mathbb{P}^1(F)}\left(
\mathbb{Z}[\mathbb{P}^1(F)\setminus\{x\}]\to \mathbb{Z}\right),
$$
with the first term sitting in degree $1$ and the differential given
by the minus the augmentation $y\mapsto -1$. This is a shifted version
of the augmented complex, hence the additional sign.
The complex $E_\bullet(F)$ is defined to be 
$$
E_\bullet(F):=\bigoplus_{x,y\in\mathbb{P}^1(F),x\neq y}\mathbb{Z}
$$
concentrated in degree $1$. Note that the index set is the set of
unordered pairs of distinct elements in $\mathbb{P}^1(F)$. 

\begin{lemma}
\label{lem:dupontexact}
There is a  map of complexes $F_0C^{\op{alt}}_\bullet(F)\to
D_\bullet(F)$  given in degree $0$ by $(x)\mapsto 1_x$ and in degree
$1$ by $(x,y)\mapsto (y)_x-(x)_y$. 

There is a map of complexes $D_\bullet(F)\to E_\bullet(F)$ given in
degree $0$ by the $0$-map and in degree $1$ by $(x)_y\mapsto
1_{(x,y)}$. 

With these maps, there is an exact sequence of complexes
$$
0\to F_0C^{\op{alt}}_\bullet(F)\otimes_{\mathbb{Z}}\mathbb{Z}[1/2]\to
D_\bullet(F)\otimes_{\mathbb{Z}}\mathbb{Z}[1/2]\to
E_\bullet(F)\otimes_{\mathbb{Z}}\mathbb{Z}[1/2]\to 0.  
$$
\end{lemma}

\begin{proof}
The first claim is simply that 
$(x,y)\mapsto (y)_x-(x)_y\mapsto -1_x+1_y$ and $(x,y)\mapsto
(y)-(x)\mapsto 1_y-1_x$ are the same map. 

The second claim is trivially true since $E_\bullet(F)$ is
concentrated in degree $1$. 

We discuss injectivity of the first map: in degree $0$, the map is
$(x)\mapsto 1_x$ which is obviously injective. In degree $1$, the map
is $(x,y)\mapsto (y)_x-(x)_y$. If $x=y$, then the image is zero, but
the element $(x,x)=-(x,x)$ is $2$-torsion, hence also $0$. Injectivity
in degree $1$ follows from this.

Surjectivity of the second map is clear, since $1_{(x,y)}$ is in the
image of $(x)_y$. 

It remains to see exactness in the middle. In degree $0$, exactness in
the middle is clear: the last term is $0$ and the first map
$(x)\mapsto 1_x$ is an isomorphism. In degree $1$, the composition is
obviously $0$, since $(x,y)\mapsto 1_{(x,y)}-1_{(y,x)}$. For an
element $c$ in $D_\bullet(F)$, the $(x,y)$-component of its image in
$E_\bullet(F)$ is trivial if the elements $(x)_y$ and $(y)_x$ occur
with opposite multiplicities. This implies exactness in the
middle. 
\end{proof}

We now can identify the homology of the complex
$F_0C^{\op{alt}}_\bullet(F)$ with the homology of the normalizer,
as in \cite[Proposition 13.22]{dupont}. 

\begin{lemma}
\label{lem:identnorm}
Let $F$ be an infinite field. 
\begin{enumerate}
\item 
The complex $D_\bullet(F)$ is acyclic. 
\item
There is an induced isomorphism
$$
\op{H}_\bullet(\op{SL}_2(F),F_0C^{\op{alt}}(F)
\otimes_{\mathbb{Z}}\mathbb{Z}[1/2])\cong 
\op{H}_\bullet(\op{N}(F),\mathbb{Z}[1/2]),
$$
where $\op{N}(F)$ is the normalizer of a maximal torus in
$\op{SL}_2(F)$. 
\end{enumerate}
\end{lemma}

\begin{proof}
(1) 
The degree $1$ part of $D_\bullet(F)$ is the $\op{SL}_2(F)$-module 
$$
\bigoplus_{x\in\mathbb{P}^1(F)}\mathbb{Z}[\mathbb{P}^1(F)\setminus\{x\}]
\cong
\op{Ind}^{\op{SL}_2(F)}_{\op{B}_2(F)}\mathbb{Z}[\mathbb{P}^1(F)\setminus\{\infty\}] 
\cong \op{Ind}^{\op{SL}_2(F)}_{\op{T}(F)}\mathbb{Z}.
$$
In particular, there is an induced isomorphism
$$
\op{H}_\bullet(\op{SL}_2(F),D_1(F))\cong
\op{H}_\bullet(\op{T}(F),\mathbb{Z}). 
$$
Similarly, we can identify the degree $0$ part of the complex as 
$$
\bigoplus_{x\in\mathbb{P}^1(F)}\mathbb{Z}\cong
\op{Ind}^{\op{SL}_2(F)}_{\op{B}_2(F)}\mathbb{Z},
$$
which yields an induced isomorphism 
$$
\op{H}_\bullet(\op{SL}_2(F),D_0(F))\cong
\op{H}_\bullet(\op{B}_2(F),\mathbb{Z}). 
$$
Since $F$ is infinite, we can identify the homology of the Borel
subgroup $\op{B}_2(F)$ with the homology of the maximal torus, i.e.,
the natural projection $\op{B}_2(F)\to\op{T}(F)$ induces an
isomorphism
$\op{H}_\bullet(\op{B}_2(F),\mathbb{Z})\cong\op{H}_\bullet(\op{T}(F),\mathbb{Z})$.
The differential
$d_1:\op{H}_\bullet(\op{T},\mathbb{Z})\to\op{H}_\bullet(\op{T},\mathbb{Z})$
is the identity, because the differential of the complex $d_1$ is
induced from the natural inclusion $\op{T}(F)\to\op{B}_2(F)$ on the
stabilizer groups. Therefore, the complex $D_\bullet(F)$ is acyclic. 

(2) 
The result will be proved using the long exact sequence of homology
groups associated to the exact sequence of complexes from 
\prettyref{lem:dupontexact}. From (1), it suffices to show that the
homology of $E_\bullet(F)$ is identified with a shift of the homology
of the normalizer. 

The action of $\op{SL}_2(F)$ on $E_1(F)=\bigoplus_{x\neq y}\mathbb{Z}$
is transitive, we can choose $\{0,\infty\}$ as orbit
representative. The stabilizer of the unordered pair $\{0,\infty\}$ is
the normalizer $\op{N}(k)$: in addition to the diagonal matrices
stabilizing both $0$ and $\infty$, the Weyl group generator 
$$
w=\left(\begin{array}{cc}
0&1\\-1&0\end{array}\right)
$$
stabilizes $\{0,\infty\}$ setwise. Therefore, there are induced
isomorphisms
$$
\op{H}_{i+1}(\op{SL}_2(F),E_\bullet(F))\cong
\op{H}_i(\op{N}(F),\mathbb{Z}).
$$
\end{proof}

From the complex $C^{\op{alt}}_\bullet(F)$, we now obtain the required
long exact sequence connecting homology of $\op{SL}_2(F)$ to the
homology of $\op{N}(F)$ and the refined scissors congruence groups
$\mathcal{RP}^1_\bullet(F)$. 

\begin{proposition}
\label{prop:rp1exact}
Let $F$ be an infinite field. 
\begin{enumerate}
\item There is a long exact sequence 
$$
\cdots\to \op{H}_\bullet(\op{N}(F),\mathbb{Z}[1/2])\to
\op{H}_\bullet(\op{SL}_2(F),\mathbb{Z}[1/2])\to
\mathcal{RP}^1_\bullet(F)\to \cdots 
$$
where $\op{N}(F)$ denotes the normalizer of a maximal torus in
$\op{SL}_2(F)$. 
\item There is a long exact sequence 
$$
\cdots\to \op{H}_\bullet(\op{N}'(F),\mathbb{Z}[1/2])\to
\op{H}_\bullet(\op{PGL}_2(F),\mathbb{Z}[1/2])\to
\mathcal{P}^1_\bullet(F)\to \cdots 
$$
where $\op{N}'(F)$ denotes the normalizer of a maximal torus in
$\op{PGL}_2(F)$. 
\item There are similar long exact sequences with
  $\mathbb{Z}/\ell$-coefficients, $\ell$ an  odd prime. 
\end{enumerate}
\end{proposition}

\begin{proof}
Consider the exact sequence of complexes
$$
0\to F_0C^{\op{alt}}_{\bullet}(F)\to C^{\op{alt}}_\bullet(F)\to
C^{\op{alt}}_\bullet(F)/F_0C^{\op{alt}}_\bullet(F)\to 0. 
$$

The equivariant homology of $F_0C^{\op{alt}}_\bullet(F)$ is identified
with the homology of the normalizer by \prettyref{lem:identnorm}. 
The equivariant homology of $C^{\op{alt}}_\bullet(F)$ is identified
with the homology of the respective group $\op{SL}_2(F)$ or
$\op{PGL}_2(F)$ by \prettyref{cor:acyclic}.
The identification of equivariant homology of the quotient
$C^{\op{alt}}_\bullet(F)/F_0C^{\op{alt}}_\bullet(F)$ with refined
scissors congruence groups is built into \prettyref{def:rp1}. 
\end{proof}

\begin{remark}
Note that the groups $\mathcal{RP}^1_\bullet(F;\mathbb{Z}/\ell)$ with
finite coefficients are not simply obtained by tensoring
$\mathcal{RP}^1_\bullet(F)\otimes_{\mathbb{Z}}\mathbb{Z}/\ell$, but
instead by tensoring the defining complex $C^{\op{alt}}_\bullet(F)$
with $\mathbb{Z}/\ell$ and then computing equivariant homology. 
\end{remark}

We make some remarks on the relation between the spectral sequences
associated to $C_\bullet(F)$ and $C^{\op{alt}}_\bullet(F)$. Note first
that there is a natural map $C_\bullet(F)\to
C^{\op{alt}}_\bullet(F)$ given by the identity on tuples of pairwise
distinct points. Both complexes are acyclic, hence their
$\op{SL}_2(F)$-equivariant homology can be identified with group
homology of $\op{SL}_2(F)$ with constant coefficients. The map between
the two complexes then provides a map between spectral spectral
sequences converging to
$\op{H}_\bullet(\op{SL}_2(F),\mathbb{Z}[1/2])$. The spectral sequence
associated to $C_\bullet(F)$ is the one usually discussed, see for
example \cite[Theorem 3.2.2]{knudson:book}, \cite[Theorem
8.19]{dupont} or \cite{hutchinson:bw}. The spectral sequence for
$C^{\op{alt}}_\bullet(F)$ as discussed above seems to appear only in
  \cite[Chapter 15]{dupont}.

Indexing the $E^1$-term as $E^1_{p,q}=\op{H}_p(\op{SL}_2(F),C_q)$ and
working as always with $\mathbb{Z}[1/2]$-coefficients, the 
$E^1$-terms of both spectral sequences are both concentrated in the
$p=0$ column as well as the two lines $q=0$ and $q=1$. From the
computations in \cite[Section 4]{hutchinson:bw}, we see that for the
complex $C_\bullet(F)$, we have 
$$
E^2_{p,0}\cong \op{H}_\bullet(\op{N}(F),\mathbb{Z}[1/2]),
$$
using that with $\mathbb{Z}[1/2]$-coefficients, the Hochschild-Serre
spectral sequence associated to the extension
$1\to\op{T}(F)\to\op{N}(F)\to\mathbb{Z}/2\to 0$ degenerates to an
isomorphism 
$$
\op{H}_\bullet(\op{N}(F),\mathbb{Z}[1/2])\cong
\op{H}_0(\mathbb{Z}/2,\op{H}_\bullet(\op{T}(F),\mathbb{Z}[1/2])). 
$$
We also saw above that for the complex $C^{\op{alt}}(F)$, we have 
$$
E^{2,\op{alt}}_{p,0}\cong\op{H}_\bullet(\op{N}(F),\mathbb{Z}[1/2]). 
$$
But from \prettyref{lem:identnorm}, it also follows that  the line
$E^{2,\op{alt}}_{p,1}$ is in fact trivial. Although the spectral
sequence for the complex $C^{\op{alt}}_\bullet(F)$ does not degenerate
at the $E^2$-term, its differentials are the maps
$\mathcal{RP}^1_q(F)\to\op{H}_{q-1}(\op{N}(F),\mathbb{Z}[1/2])$, and
all information is contained in the long exact sequence of
\prettyref{prop:rp1exact}. 

We see, using the comparison map $C_\bullet(F)\to
C^{\op{alt}}_\bullet(F)$, that the differentials
$d^{q-1}:E^{q-1}_{0,q}\to E^{q-1}_{q-2,1}$ are surjective (with
$\mathbb{Z}[1/2]$-coefficients). In the $E^\infty$-term, we then find
$E^\infty_{q,1}\cong E^{\infty,\op{alt}}_{q,1}\cong 0$. Moreover, we
can identify the $\mathcal{RP}^1_q(F)$ groups as the entries
$E^q_{0,q}$. The connecting map
$\mathcal{RP}^1_q(F)\to\op{H}_{q-1}(\op{N}(F),\mathbb{Z}[1/2])$ is the
differential $d^q:E^q_{0,q}\to E^q_{q-1,0}$. In particular, the groups
$\mathcal{RP}^1_q(F)$ can be identified (in the spectral sequence for
$C_\bullet(F)$) with the kernel of the two differentials
$d^1:E^1_{0,q}\to E^1_{0,q-1}$ and $d^{q-1}:E^{q-1}_{0,q}\to
E^{q-1}_{q-2,1}$. 

\subsection{Low-dimensional computations}

We note some further well-known statements on the (refined) scissors
congruence groups, cf. \cite{dupont} and
\cite{hutchinson:bw}.  

\begin{proposition}
Let $F$ be an infinite field. 
\begin{enumerate}
\item $\mathcal{RP}^1_q(F)=0$ for $q\leq 1$. 
\item $\mathcal{P}^1_q(F)=0$ for $q\leq 2$.
\item $\mathcal{RP}^1_2(F)=\op{I}(F)$, where $I$ denotes the
  fundamental ideal in the Witt ring of $F$. 
\end{enumerate}
\end{proposition}

\begin{proof}
By definition $\mathcal{RP}^1_q(F)$ is the $\op{SL}_2(F)$-equivariant
homology of the quotient
$C^{\op{alt}}_\bullet(F)/F_0C^{\op{alt}}_\bullet(F)$. Since
$\op{SL}_2(F)$ acts transitively on the set of pairs of distinct
points in $\mathbb{P}^1(F)$, the complex $C^{\op{alt}}_\bullet(F)$ has
one copy of $\mathbb{Z}$ in degree $0$ and one copy of $\mathbb{Z}$ in
degree $1$. But these are both degenerate configurations, so they are
contained in the subcomplex $F_0C^{\op{alt}}_\bullet(F)$. This proves
the first claim (and part of the second claim). 

(2) follows similarly, using that the action of
$\op{PGL}_2(F)$ on $3$ distinct points in $\mathbb{P}^1(F)$ is also
transitive. The result is a single copy of $\mathbb{Z}$ in degree $2$
of the complex $\op{PGL}_2(F)\backslash C^{\op{alt}}_\bullet(F)$,
which is killed by the differential from degree $3$. 

(3) is a reformulation of the computations in \cite[Section
4.5]{hutchinson:bw}. 
\end{proof}

We want to note that the low-degree part of the long exact sequence
from \prettyref{prop:rp1exact} is the Bloch-Wigner sequence proved in
\cite{hutchinson:bw}. In the notation of this appendix, it reads (all
groups with $\mathbb{Z}[1/2]$-coefficients):
$$
0\to \op{H}_3(\op{N}(F))\to \op{H}_3(\op{SL}_2(F))\to
\mathcal{RP}^1_3(F)\to \op{H}_2(\op{N}(F))\to
\op{H}_2(\op{SL}_2(F))\to 0.
$$

There would be a lot more to say about the refined scissors congruence
groups. We only make two final remarks. The relation of
$\op{H}_3(\op{SL}_2(\mathbb{C}))$ to actual scissors congruences in
$S^3$ and $\mathbb{H}^3$ is the subject of the book \cite{dupont}. The
Friedlander-Milnor conjecture is equivalent to unique divisibility of
the scissors congruence groups $\mathcal{P}^1_\bullet(F)$. This is
known so far only for $\mathcal{P}^1_3(F)$ by the work of Suslin,
Dupont and Sah.


\begin{thebibliography}{Hut11b}

\bibitem[AB08]{abramenko:brown}
P. Abramenko and K.S. Brown.  Buildings. Graduate Texts in Mathematics
248. Springer, 2008. 

\bibitem[Ati56]{atiyah:ks}
M.F. Atiyah. On the Krull-Schmidt theorem with applications to
sheaves. Bull. Soc. Math. France 84 (1956), 307--317.

\bibitem[Ati57]{atiyah:connection}
M.F. Atiyah. Complex analytic connections in fibre
bundles. Trans. Amer. Math. Soc. 85 (1957), no. 1, 181--207.

\bibitem[Bro94]{brown:book}
K.S. Brown. Cohomology of groups. Corrected reprint of the 1982
original. Graduate Texts in Mathematics, 87. Springer, 1994. 

\bibitem[Dup01]{dupont}
J.L. Dupont. Scissors congruences, group homology and characteristic
classes. Nankai Tracts in Mathematics 1. World Scientific, 2001.

\bibitem[FM84]{friedlander:mislin}
E.M. Friedlander and G. Mislin. Cohomology of classifying spaces of
complex Lie groups and related discrete
groups. Comment. Math. Helv. 59 (1984), 347--361.

\bibitem[Har77]{hartshorne}
R. Hartshorne. Algebraic geometry. Graduate Texts in Mathematics
52. Springer, 1977.

\bibitem[Hut11a]{hutchinson:rb}
K. Hutchinson. A refined Bloch group and the third homology of $SL_2$
of a field. J. Pure Appl. Algebra 217 (2013), 2003--2035.

\bibitem[Hut11b]{hutchinson:bw}
K. Hutchinson. A Bloch-Wigner complex for $SL_2$. 
J. K-theory 12 (2013), no. 1, 15--68.

\bibitem[Knu01]{knudson:book}
K.P. Knudson. Homology of linear groups. 
Progress in Mathematics, 193. Birkhäuser Verlag, Basel, 2001.

\bibitem[Mis94]{mislin:tate}
G. Mislin. Tate cohomology for arbitrary groups via
satellites. Topology Appl. 56 (1994), no. 3, 293--300. 

\bibitem[Rah13]{rahm:transactions}
A.D. Rahm. The homological torsion of $\op{PSL}_2$ of the imaginary
quadratic integers. Trans. Amer. Math. Soc. 365 (2013), no. 3,
1603--1635. 

\bibitem[Rah14]{rahm:jalg}
A.D. Rahm. Accessing the cohomology of discrete groups above their
virtual cohomological dimension. J. Alg. 404 (2014), no. C, 152--175. 

\bibitem[Ros73]{rosen:sunits}
M. Rosen. $S$-units and $S$-class group in algebraic function
fields. J. Algebra 26 (1973), 98--108.

\bibitem[Ser80]{serre:book}
J.-P. Serre. Trees. Springer, 1980.

\bibitem[Stu76]{stuhler:76}
U. Stuhler. Zur Frage der endlichen Pr{\"a}sentierbarkeit gewisser
arithmetischer Gruppen im Funktionenk{\"o}rperfall. 
Math. Ann. 224 (1976), no. 3, 217--232. 

\bibitem[Stu80]{stuhler:80}
U. Stuhler. Homological properties of certain arithmetic groups in the
function field case. Invent. Math. 57 (1980), no. 3, 263--281.

\bibitem[Sus90]{suslin:k3}
A.A. Suslin. $K_3$ of a field and the Bloch group. Trudy Mat. Inst. Steklov., 
183:190-199, 229, 1990. Translated in Proc. Steklov Inst. Math. 1991, no. 4,
217-239, Galois Theory, rings, algebraic groups and their applications (Russian)

\end{thebibliography}
\end{document}